\documentclass[a4paper]{amsart}

\usepackage[utf8]{inputenc} 

\usepackage{amsmath, amssymb, amsthm, mathtools, mathrsfs} 
\usepackage{bm} 
\numberwithin{equation}{section} 

\usepackage{tikz-cd} 

\usepackage[top=2.3cm, bottom=2.3cm, left=2.3cm, right=2.3cm]{geometry} 
\usepackage{setspace} 
\usepackage[shortlabels]{enumitem} 

\usepackage{xcolor} 

\usepackage{hyperref} 
\definecolor{urlcolor}{rgb}{0,0,0}
\definecolor{linkcolor}{rgb}{.7,0.10,0.2}
\definecolor{citecolor}{rgb}{.12,.54,.11}
\hypersetup{
    breaklinks=true,     
    colorlinks=true,     
    urlcolor=urlcolor,   
    linkcolor=linkcolor, 
    citecolor=citecolor  
}

\usepackage{cleveref}

\usepackage{thmtools}

\usepackage[
    style=alphabetic,     
    doi=false,
    isbn=false,
    url=false,
    eprint=false,
    block=space,          
    maxcitenames=100,     
    maxbibnames=100,      
    sorting=anyt          
]{biblatex}
\usepackage[normalem]{ulem} 

\usepackage{quiver}

\declaretheoremstyle[
    spaceabove=9pt, spacebelow=9pt,
    postheadspace=.5em,
    headfont=\normalfont\bfseries,
    headpunct={},
    headformat={\NUMBER.\@\NOTE},
    notefont=\normalfont\bfseries\boldmath,
    notebraces={}{.},
]{para}

\theoremstyle{plain}

\newtheorem{theorem}{Theorem}[section]

\newtheorem{proposition}[theorem]{Proposition}
\newtheorem{lemma}[theorem]{Lemma}
\newtheorem{conjecture}[theorem]{Conjecture}

\newtheorem{corollary}[theorem]{Corollary}

\theoremstyle{definition}
\newtheorem{remark}[theorem]{Remark}
\newtheorem{example}[theorem]{Example}

\newtheorem*{ACK}{Acknowledgement}

\newtheorem{definition}[theorem]{Definition}
\theoremstyle{para}
\newtheorem{para}[theorem]{}

\crefformat{enumi}{{}#2{\upshape#1}#3}
\crefrangeformat{enumi}{{}{\upshape#3#1#4--#5#2#6}}
\crefmultiformat{enumi}{{}{\upshape#2#1#3}}{ and {\upshape#2#1#3}}{, {\upshape#2#1#3}}{, and~{\upshape#2#1#3}}

\crefformat{para}{{}#2{\upshape\S#1}#3}
\crefrangeformat{para}{{}{\upshape#3\S\S#1#4--#5#2#6}}
\crefmultiformat{para}{{}{\upshape#2\S\S#1#3}}{ and {\upshape#2#1#3}}{, {\upshape#2#1#3}}{, and~{\upshape#2#1#3}}

\numberwithin{equation}{section}

\newcommand{\Aut}{\mathrm{Aut}}  

\renewcommand{\bar}{\overline}  
\renewcommand{\setminus}{\smallsetminus}  

\newcommand{\BA}{\mathbb{A}}

\newcommand{\BG}{\mathbb{G}}

\newcommand{\BN}{\mathbb{N}}

\newcommand{\BQ}{\mathbb{Q}}

\newcommand{\BZ}{\mathbb{Z}}

\newcommand{\CF}{\mathcal{F}}
\newcommand{\CG}{\mathcal{G}}
\newcommand{\CH}{\mathcal{H}}

\newcommand{\CL}{\mathcal{L}}

\newcommand{\CS}{\mathcal{S}}

\newcommand{\CU}{\mathcal{U}}

\newcommand{\CW}{\mathcal{W}}
\newcommand{\CX}{\mathcal{X}}
\newcommand{\CY}{\mathcal{Y}}
\newcommand{\CZ}{\mathcal{Z}}

\newcommand{\SL}{\mathscr{L}}

\newcommand{\SP}{\mathscr{P}}

\newcommand{\rmB}{\mathrm{B}}  
\newcommand{\BPS}{\mathcal{BPS}}  
\newcommand{\GL}{\mathrm{GL}}  
\newcommand{\Fg}{\mathfrak{g}} 
\newcommand{\rmH}{\mathrm{H}}  
\newcommand{\Spec}{\mathrm{Spec}}  
\newcommand{\Sym}{\mathrm{Sym}}  

\newcommand{\rmc}{\mathrm{c}}
\newcommand{\cms}{/\!\!/}

\newcommand{\Face}{\mathsf{Face}}
\newcommand{\pH}{{^\mathrm{p}\!\CH}}
\newcommand{\ICS}{\mathcal{IC}}
\newcommand{\id}{\mathrm{id}}  
\newcommand{\IH}{\mathrm{IH}} 
\newcommand{\gr}{\mathrm{gr}}  
\newcommand{\pt}{\mathrm{pt}}
\newcommand{\rmm}{\mathrm{m}}
\newcommand{\nd}{\mathrm{nd}}
\newcommand{\sgn}{\mathrm{sgn}}
\renewcommand{\sp}{\mathrm{sp}}

\renewcommand{\Spec}{\mathrm{Spec}}  
\newcommand{\Tan}{\mathrm{T}}
\newcommand{\Ft}{\mathfrak{t}}
\newcommand{\vir}{\mathrm{vir}} 
\newcommand{\rmX}{\mathrm{X}}
\renewcommand{\gr}{\mathrm{gr}}
\newcommand{\ev}{\mathrm{ev}}
\DeclareMathOperator{\vdim}{vdim}
\newcommand{\GIT}{{/\! \! /}}
\renewcommand{\Im}{\operatorname{Im}}
\newcommand{\cofib}{\operatorname{cofib}}

\hypersetup{pdftitle=The BPS decomposition theorem}

\title{The BPS decomposition theorem}
\date{\today}

\author{Lucien Hennecart}

\address{Laboratoire Ami\'enois de Math\'ematique Fondamentale et Appliqu\'ee, CNRS UMR 7352, Universit\'e de Picardie Jules Verne, 33 rue Saint Leu, 80000 Amiens, France}

\email{lucien.hennecart@u-picardie.fr}

\author{Tasuki Kinjo}

\address{Research Institute for Mathematical Sciences, Kyoto University, Kyoto 606-8502, Japan}

\email{tkinjo@kurims.kyoto-u.ac.jp}

\begin{document}

\begin{abstract}

    We prove the BPS decomposition theorem (a.k.a. cohomological integrality theorem) decomposing the cohomology of smooth symmetric stacks into the Weyl-invariant part of the cohomological Hall induction of the intersection cohomology of good moduli spaces. 
    As a consequence, we establish the BPS decomposition theorem for the Borel--Moore homology of $0$-shifted symplectic stacks and for the critical cohomology of symmetric $(-1)$-shifted symplectic stacks, thereby generalizing the main result of Bu--Davison--Ibáñez Nuñez--Kinjo--Pădurariu to the non-orthogonal setting.

    We will present three applications of our main result.
    First, we confirm Halpern-Leistner's conjecture on the purity of the Borel--Moore homology of $0$-shifted symplectic stacks admitting proper good moduli spaces, extending Davison's work on the moduli stack of objects in $2$-Calabi--Yau categories.
    Second, we prove versions of Kirwan surjectivity for the critical cohomology of symmetric $(-1)$-shifted symplectic stacks and for the Borel--Moore homology of $0$-shifted symplectic stacks.
    Finally, by applying our main result to the character stacks associated with compact oriented $3$-manifolds, we reduce the quantum geometric Langlands duality conjecture for $3$-manifolds, as formulated by Safronov, from an isomorphism between infinite-dimensional critical cohomologies to a comparison of finite-dimensional BPS cohomologies.
    \end{abstract}

\maketitle

\setcounter{tocdepth}{1}
\tableofcontents

\section{Introduction}
\label{section:introduction}
In algebraic geometry, moduli stacks parametrize objects of interest in a way that reflects symmetries. One considers for example moduli stacks of coherent sheaves on smooth projective varieties, of representations of algebras in representation theory, or of local systems on algebraic varieties. Historically, it has proven useful to produce algebraic varieties from moduli problems, using Geometric Invariant Theory (GIT) \cite{mumford1994geometric}. Nowadays, it is understood that GIT produces meaningful approximations of the corresponding moduli stacks \cite{alper2013good}.

The cohomology of moduli spaces is the host of rich enumerative invariants that have drawn the attention of many studies. In algebraic geometry, refined Donaldson--Thomas invariants \cite{thomas2000holomorphic} are cohomological invariants. In representation theory, Kac polynomials of quivers are also cohomological invariants, which can be identified with refined BPS invariants of a $3$-Calabi--Yau category built out of the quiver \cite{hausel2013positivity,davison2018purity}.

In this paper, we prove the BPS decomposition theorem for symmetric stacks, completing the study of \cite{hennecart2024cohomological,hennecart2024cohomological2} and \cite{bu2025cohomology}. The BPS decomposition bridges the cohomology of a large class of stacks of interest to the intersection cohomology of the good moduli space, and can be regarded as an effective form of the BBDG decomposition theorem for the good moduli space morphisms.

The notion of the BPS decomposition originally appears in the context of cohomological Hall algebras, in the work of Kontsevich and Soibelman \cite{kontsevich2011cohomological}. In this situation, it amounts to the boundedness of some invariants. At the level of vector spaces, it takes the form of the PBW theorem, i.e.,  an identification of the underlying vector space of the cohomological Hall algebra with the symmetric power of another vector space. It has been studied thoroughly, in particular in \cite{efimov2012cohomological} for symmetric quivers, in \cite{davison2020cohomological} for quivers with potential, in \cite{davison2023bps} for general $2$-Calabi--Yau categories, including the categories of coherent sheaves on K3 and Abelian surfaces, and in \cite{bu2025cohomology} for general $3$-Calabi--Yau categories satisfying the appropriate geometric conditions.

The BPS decomposition for symmetric stacks have recently been introduced from diverse perspectives. In \cite{hennecart2024cohomological}, we prove the BPS decomposition theorem for stacks of the form $V/G$ where $V$ is a  symmetric representation of a reductive group in a purely combinatorial way. This decomposition involves some finite-dimensional vector spaces $P_{\lambda}$ indexed by cocharacters $\lambda$ of a maximal torus $T$ of $G$ -- the \emph{algebraic BPS cohomologies} of $V/G$. We did not give a geometric interpretation of these $P_{\lambda}$. In \cite{hennecart2024cohomological2}, we proved a version of the BPS decomposition theorem at the sheafified level, which holds for quotient stacks $X/G$ of a smooth affine $G$-variety ($G$ a reductive group). The vector spaces $P_{\lambda}$ are replaced by \emph{cohomologically graded complexes of monodromic mixed Hodge modules} which are denoted $\BPS_{X/G,\lambda}$. More recently \cite{bu2025cohomology}, the BPS decomposition theorem has been generalized to any smooth stack having a good moduli space and satisfying some appropriate assumptions (see \S\ref{para:assumptions-stacks}). The moduli stacks dealt with are \emph{orthogonal stacks}, a strengthening of the notion of symmetric stacks. In \cite{bu2025cohomology}, it is proven directly that in the orthogonal case, the BPS decomposition involves the intersection cohomology and not undetermined complexes. This implies for example that, in the orthogonal case, the vector spaces $P_{\lambda}$ of \cite{hennecart2024cohomological} are given (when non-zero) by the intersection cohomology of the affine GIT quotients.

In the present paper, we combine a certain vanishing bound for the vector spaces $P_{\lambda}$ given in \cite{hennecart2024cohomological} with the vanishing cycle functoriality of the cohomological Hall integrality map of \cite{bu2025cohomology}, together with a fine study of the asymptotic behaviour of the Hall induction to prove a strong version of the BPS decomposition theorem for any symmetric stacks satisfying the geometric assumptions \Cref{para:assumptions-stacks}. As a consequence, we obtain the BPS decomposition for any $0$-shifted symplectic stacks and we extend the scope of the BPS decomposition theorem for $(-1)$-shifted symplectic stacks in \cite{bu2025cohomology} from orthogonal to more general symmetric stacks. As a special case, we obtain the BPS decomposition for the stack of $G$-local systems on arbitrary compact oriented $3$-manifolds, a case that was left open in \cite{bu2025cohomology} due to the lack of obvious orthogonal structure. This case is relevant to the Langlands-type conjecture for character stacks of $3$-manifolds (\Cref{conj-3mfd}).

\subsection{Main results}
We now briefly present our main results. Let $V$ be a representation of a (possibly disconnected) reductive group $G$. We let $T\subset G^{\circ}$ be a maximal torus of the neutral component of $G$ and we denote by $\CW(V)$ the collection of $T$-weights of $V$ counted with multiplicities. 

\begin{definition}
\label{definition:notions_symmetric_stacks}
Let $V$ be a representation of a (possibly disconnected) reductive group $G$. We say that
 \begin{enumerate}
  \item $V$ is a \emph{self-dual} or \emph{symmetric} representation if $V\cong V^*$ as $G$-representations.
  \item $V$ is an \emph{almost self-dual} or \emph{almost symmetric} representation if $V\cong V^*$ as $G^{\circ}$-representations, equivalently if $\CW(V)=\CW(V^*)$.
 \end{enumerate}

Let $\CX$ be an algebraic stack. We say that
\begin{enumerate}
 \item $\CX$ is \emph{self-dual} or \emph{symmetric} if for any closed point $x\in\CX$, the tangent space $\Tan_{\CX,x}$ is a self-dual representation of the stabilizer group $G_x$.
 \item $\CX$ is \emph{almost self-dual} or \emph{almost symmetric} if for any closed point $x\in\CX$, the tangent space $\Tan_{\CX,x}$ is an almost self-dual representation of $G_x$ (i.e. a self-dual representation of $G_x^{\circ}$).
\end{enumerate}
\end{definition}

Clearly, in \Cref{definition:notions_symmetric_stacks}, we have the implication $(1)\implies (2)$.

In this paper, we use freely the notion of \emph{component lattices} of algebraic stacks introduced in \cite{bu2025intrinsic} and summarized in \Cref{ssec-component-lattice}. This notion extends that of flats and chambers for quotient stack of the form $V/G$ where $V$ is a representation of a reductive group $G$, see for example \cite[\S3]{hennecart2024cohomological} for this viewpoint. In particular, we may use the categories of \emph{faces} $\Face(\CX)$, \emph{non-degenerate faces} $\Face^{\nd}(\CX)$ and \emph{special faces} $\Face^{\sp}(\CX)$ of the rationalized component lattice $\mathrm{CL}_{\BQ}(\CX)$ of $\CX$.

\begin{para}[Assumptions on stacks]
\label{para:assumptions-stacks}
In this paper, we keep the running assumptions of \cite[\S1.2.5]{bu2025cohomology} regarding stacks.
Namely, we let $\CX$ be an Artin stack satisfying the following assumptions:

\begin{enumerate}
 \item $\CX$ has affine diagonal and a good moduli space $p\colon\CX\rightarrow X$.
 \item $\CX$ has quasi-compact connected components and quasi-compact graded points.
 \item $\CX$ is almost symmetric.
 \item For each non-degenerate face $(F,\alpha)\in\Face^{\nd}(\CX)$, there exists a global equivariant parameter for $\CX_{\alpha}$; see \S \ref{ssec-global-equiv-parameter} for the detail.
\end{enumerate}
\end{para}

For any $(F,\alpha)\in\Face(\CX)$, there is a natural commutative square
\[\begin{tikzcd}
	{\CX_{\alpha}} & \CX \\
	{X_{\alpha}} & X
	\arrow[from=1-1, to=1-2]
	\arrow["{p_{\alpha}}"', from=1-1, to=2-1]
	\arrow["p", from=1-2, to=2-2]
	\arrow["{g_{\alpha}}"', from=2-1, to=2-2]
\end{tikzcd}\]
connecting the respective good moduli spaces $X$ and $X_{\alpha}$ of $\CX$ and $\CX_{\alpha}$.
\begin{theorem}
 \label{theorem:main_theorem}
Let $\CU$ be a smooth algebraic stack satisfying the assumptions \Cref{para:assumptions-stacks}. Then there exists an isomorphism of monodromic mixed Hodge module complexes on $U$:
\begin{equation}\label{eq-intro-main}
 \bigoplus_{(F,\alpha)\in\Face^{\sp}(\CU)}(g_{\alpha,*}\ICS^{\circ}_{U_{\alpha}}\otimes\rmH^*(\rmB\BG_{\rmm}^{\dim F})_{\vir}\otimes\sgn_{\alpha})^{\Aut(\alpha)}\cong p_*\ICS_{\mathcal{U}}\,,
\end{equation}
where the direct sum runs over isomorphism classes of special faces, 
\[
 \ICS_{U_{\alpha}}^{\circ}=\left\{
 \begin{aligned}
  &\ICS_{U_{\alpha}}\quad\text{if the neutral component of the generic stabilizer of a closed point of $\CX_{\alpha}$ is isomorphic to $\BG_{\rmm}^{\dim F}$} \\
  &0\quad\text{otherwise}\,,
 \end{aligned}\right.
\]
$\rmH^*(\rmB\BG_{\rmm}^{\dim F})_{\vir}\coloneqq \rmH^*(\rmB\BG_{\rmm}^{\dim F})\otimes\SL^{\dim F/2}$, $\SL$ is the Tate twist, and $\sgn_{\alpha}$ denotes the cotangent sign representation, \cite[\S4.4.5]{bu2025cohomology}.
In particular, we have an isomorphism 
\[
	\bigoplus_{(F,\alpha)\in\Face^{\sp}(\CU)} (\IH^{\circ} (U_{\alpha})\otimes\rmH^*(\rmB\BG_{\rmm}^{\dim F})_{\vir}\otimes\sgn_{\alpha})^{\Aut(\alpha)}\cong \mathrm{H}^*(\mathcal{U})_{\mathrm{vir}},
\]
where $\IH^{\circ} (U_{\alpha})$ and $\mathrm{H}^*(\mathcal{U})_{\mathrm{vir}}$ are the result of the derived global section functor applied to $\mathcal{IC}^{\circ}_{U_{\alpha}}$ and $\mathcal{IC}_{\mathcal{U}}$ respectively.
\end{theorem}

\begin{remark}\label{rmk-T-action}
	If there is an auxiliary torus $T'$ acting on $\mathcal{U}$, it acts naturally on $\mathcal{U}_{\alpha}$ and $U_{\alpha}$, and the BPS decomposition \eqref{eq-intro-main} is a decomposition of $T'$-equivariant complexes over $U$.
	This follows from the fact that the PBW map is induced by the cohomological Hall induction (see \S \ref{ssec-cohi}), which is clearly $T'$-equivariant by construction.
	In particular, we obtain a natural isomorphism
	\[
	\bigoplus_{(F,\alpha)\in\Face^{\sp}(\CU)} (\IH_{T'}^{\circ} (U_{\alpha})\otimes\rmH^*(\rmB\BG_{\rmm}^{\dim F})_{\vir}\otimes\sgn_{\alpha})^{\Aut(\alpha)}\cong \mathrm{H}_{T'}^*(\mathcal{U})_{\mathrm{vir}},
	\]
	where $\IH_{T'}^{\circ}(U_{\alpha})$ denotes the $T'$-equivariant intersection cohomology of $U_{\alpha}$ if the neutral component of the generic stabilizer of a closed point of $\CU_{\alpha}$ is isomorphic to $\BG_{\rmm}^{\dim F}$ and $\IH_{T'}^{\circ}(U_{\alpha})=0$ otherwise, and $\mathrm{H}_{T'}^*(\mathcal{U})_{\mathrm{vir}}$ is the $T'$-equivariant cohomology of $\CU$.
\end{remark}

We now provide a more explicit form of \Cref{theorem:main_theorem} for quotient stacks.
Let $V$ be an almost symmetric representation of a connected reductive group $G$. We let $T$ be a maximal torus of $G$, $W$ be the Weyl group of $G$ and $\SP_{V}$ the quotient of the lattice of cocharacters $\Lambda_T$ by the equivalence relation $\lambda\sim\mu\iff (V^{\lambda}=V^{\mu}\text{ and } L_{\lambda}=L_{\mu}$), where $L_{\lambda}$ denotes the centralizer of $\lambda$, a Levi subgroup of $G$. The class of $\lambda\in\Lambda_T$ in $\SP_{V}$ is denoted by $\overline{\lambda}$. We have a commutative diagram
\[\begin{tikzcd}
	{V^{\lambda}/L_{\lambda}} & {V/G} \\
	{V^{\lambda}\cms L_{\lambda}} & {V\cms G}
	\arrow[from=1-1, to=1-2]
	\arrow["{p_{\lambda}}"', from=1-1, to=2-1]
	\arrow["p", from=1-2, to=2-2]
	\arrow["{g_{\lambda}}"', from=2-1, to=2-2]\,.
\end{tikzcd}\]

For $\lambda\in\Lambda_T$, we let $W_{\lambda}\coloneqq \{w\in W\mid \overline{w\cdot\lambda}=\overline{\lambda}\}/W^{\lambda}$ be the centralizer of $\overline{\lambda}$ in $W$ modulo the Weyl group $W^{\lambda}$ of $L_{\lambda}$.

\begin{theorem}
 \label{theorem:local_case}
Let $V$ be a  symmetric representation of a connected reductive group $G$. Then, there exists an isomorphism
\[
 \bigoplus_{\tilde{\lambda}\in\SP_V/W}(g_{\lambda,*}\ICS_{V^{\lambda}\cms L_{\lambda}}^{\circ}\otimes \rmH^*_{G_{\lambda}}(\pt)_{\vir}\otimes\varepsilon_{V/G,\lambda})^{W_{\lambda}}\cong p_*\ICS_{V/G}
\]
where
\[
 G_{\lambda}=\mathrm{Z}(L_{\lambda})\cap\ker(L_{\lambda}\rightarrow\GL(V^{\lambda}))
\]
is the intersection of the center of $L_{\lambda}$ with the kernel of the action of $L_{\lambda}$ on $V^{\lambda}$, $\varepsilon_{V/G,\lambda}\colon W_{\lambda}\rightarrow\{\pm1\}$ denotes the cotangent sign representation of $V/G$ (also defined purely combinatorially in \cite[Proposition~4.6]{hennecart2024cohomological})
and
\[
\ICS_{V^{\lambda}\cms L_{\lambda}}^{\circ}=\left\{
 \begin{aligned}
&\ICS_{V^{\lambda}\cms L_{\lambda}}\quad\text{if a general closed orbit in $V^{\lambda}$ has finite stabilizer in $L_{\lambda}/G_{\lambda}$\,,}\\
&0\quad\text{otherwise}\,.
 \end{aligned}\right.
\]

\end{theorem}

\begin{remark}
\Cref{theorem:local_case} for symmetric representations of reductive groups is a particular case of \Cref{theorem:main_theorem}, namely the case $\CU=V/G$. However, it turns out that \Cref{theorem:local_case} implies \Cref{theorem:main_theorem} via \'etale slices and an induction, and this is precisely how we prove \Cref{theorem:main_theorem} in \Cref{sec-proof-main}.
\end{remark}

As an application of \Cref{theorem:main_theorem}, we will prove $(-1)$- and $0$-shifted symplectic versions of the BPS decomposition theorem in \Cref{thm-main--1}.
Here, we explain the statement, specializing to the character stacks of $3$-manifolds for the sake of the exposition.
Let $M$ be a compact oriented $3$-manifold and $G$ be a connected reductive algebraic group.
We let $\mathcal{L}\mathrm{oc}_G(M)$ denote the moduli stack of $G$-local systems on $M$ and $\mathrm{Loc}_G(M)$ be its good moduli space.
It follows from \cite{pantev2013shifted}, \cite{ben2015darboux} and \cite{naef2023torsion} that there exists a canonical perverse sheaf $\varphi_{\mathcal{L}{\mathrm{oc}}_G(M)}$ on $\mathcal{L}{\mathrm{oc}}_G(M)$.
We define the critical cohomology of $\mathcal{L}{\mathrm{oc}}_G(M)$ as
\[
 \mathrm{H}^*_{\mathrm{crit}}(\mathcal{L}{\mathrm{oc}}_G(M)) \coloneqq \mathrm{H}^*(\mathcal{L}{\mathrm{oc}}_G(M), \varphi_{\mathcal{L}{\mathrm{oc}}_G(M)}).	
\]
The $(-1)$-shifted version of the BPS decomposition theorem implies the following:
\begin{theorem}\label{thm-characterstack-intro}
	For any Levi subgroup $L$ of $G$, there exists a finite dimensional subspace $\mathrm{H}^*_{\mathrm{BPS}}({\mathrm{Loc}}_L(M)) \subset  \mathrm{H}^*_{\mathrm{crit}}(\mathcal{L}{\mathrm{oc}}_L(M))$, the relative Weyl group $W_G(L)$ acts on the BPS cohomology $\mathrm{H}^*_{\mathrm{BPS}}({\mathrm{Loc}}_L(M))$, and there exists a natural isomorphism
    \[
     \bigoplus_{\substack{L \subset G : \\ \textnormal{Levi subgroups}}} \left( \mathrm{H}^*_{\mathrm{BPS}}({\mathrm{Loc}}_L(M)) \otimes \mathrm{H}^*( \mathrm{B} \mathrm{Z}(L))_{\mathrm{vir}} \right)^{W_G(L)} \cong  \mathrm{H}^*_{\mathrm{crit}}(\mathcal{L}{\mathrm{oc}}_G(M)),
    \]
where the direct sum runs over Levi subgroups of $G$ up to conjugation.
\end{theorem}
We refer to \Cref{thm:main-characterstack} for the precise statement.
The map from left to right is defined by the $3$-manifold analogue of the Eisenstein series as in \cite[Corollary 9.12]{kinjo2024cohomological}.
By \Cref{thm-characterstack-intro}, we can regard $\mathrm{H}^*_{\mathrm{BPS}}({\mathrm{Loc}}_G(M))$ as a finite-dimensional model for the critical cohomology.
It has an application to the quantum geometric Langlands duality for $3$-manifolds, which is formulated as follows:
\begin{conjecture}[Safronov]\label{conj-3mfd}
	Let $G$ be a connected reductive group and $G^{\vee}$ be the Langlands dual of $G$. Then there is a natural isomorphism
	\[
		\mathrm{H}^*_{\mathrm{crit}}(\mathcal{L}{\mathrm{oc}}_G(M)) \cong 		\mathrm{H}^*_{\mathrm{crit}}(\mathcal{L}{\mathrm{oc}}_{G^{\vee}}(M)).
	\]
\end{conjecture}

\Cref{thm-characterstack-intro} reduces this conjecture to the comparison of the more tractable finite-dimensional BPS cohomologies:

\begin{corollary}
	To prove \Cref{conj-3mfd}, it is enough to prove an isomorphism
	\[
	\mathrm{H}^*_{\mathrm{BPS}}({\mathrm{Loc}}_G(M)) \cong 		\mathrm{H}^*_{\mathrm{BPS}}({\mathrm{Loc}}_{G^{\vee}}(M))
	\]
	as representations over the subgroups of the outer automorphism group of $G$ preserving the $(-1)$-shifted symplectic structures.
\end{corollary}

Now we turn our attention to $0$-shifted symplectic stacks.
By the dimensional reduction theorem proved by the second author \cite[Theorem 4.14]{kinjo2022dimensional},
the $(-1)$-shifted symplectic version of the BPS decomposition theorem implies a version for $0$-shifted symplectic stacks: see \Cref{thm-main0} for the precise statement.
As an application, we will deduce Halpern-Leistner's purity conjecture \cite[Conjecture 4.4]{halpern2015theta}:

\begin{theorem}
	Let $\mathcal{Y}$ be a $0$-shifted symplectic stack with proper good moduli space.
	Then the Borel--Moore homology $\mathrm{H}^{\mathrm{BM}}_*(\mathcal{Y})$ is pure.
\end{theorem}

As another application of our main results, we obtain Kirwan surjectivity statements for the critical cohomology of symmetric $(-1)$-shifted symplectic stacks and $0$-shifted symplectic stacks: see \Cref{ssec-Kirwan} for details.

\subsection{Notations and conventions}
\begin{enumerate}
 \item Throughout the paper, the base-field of geometric objects is the complex number field.
 \item A derived algebraic stack is a derived Artin stack over the complex number field which is quasi-separated, has separated inertia, and affine stabilizers.
 \item Reductive groups are not assumed to be connected. The neutral component of a reductive algebraic group $G$ is denoted by $G^{\circ}$. The center of $G$ is denoted by $\mathrm{Z}(G)$. If $T$ is a torus, $\Lambda_T$ denotes its cocharacter lattice.
 \item Cohomology and intersection cohomology, as well as constructible sheaves, are considered with rational coefficients.
 \item The intersection complex of an algebraic stack $\CY$ is denoted by $\ICS_{\CY}$.
 \item The Tate mixed Hodge structure is denoted by $\SL\coloneqq \rmH^*_{\rmc}(\BA^1,\BQ)$.
  \item The intersection cohomology of $\CY$ is denoted by $\IH^*(\CY)$. It is given the perverse shift, which means that if $\CY$ is an algebraic variety, $\IH^i(\CY)$ may be non-zero only for $-\dim \CY\leq i\leq \dim \CY$. If $T'$ is a torus acting on $\CY$, we let $\IH^*_{T'}(\CY)$ be the $T'$-equivariant intersection cohomology of $\CY$.
 \item For a smooth algebraic stack $\mathcal{U}$, we define $\mathrm{H}^*(\mathcal{U})_{\mathrm{vir}} \coloneqq \mathrm{H}^*(\mathcal{U}) \otimes \mathscr{L}^{- \dim \CU / 2}$. It coincides with the intersection cohomology, but we use this notation to stress that the stack we consider is smooth and indicate that it is isomorphic to the ordinary cohomology up to shifts and twists. If $\CU$ has a $T'$-action for some torus, we define $\mathrm{H}_{T'}^*(\mathcal{U})_{\mathrm{vir}} \coloneqq \mathrm{H}_{T'}^*(\mathcal{U}) \otimes \mathscr{L}^{- \dim \CU / 2}$.
 \item The perverse cohomology functors are denoted by $\pH^i$ for $i\in\BZ$. If $\CF$ is a complex of constructible sheaves, we let $\pH(\CF)\coloneqq \bigoplus_{i\in\BZ}\pH^i(\CF)[-i]$ be its total cohomology. If $f\colon\CF\rightarrow\CG$ is a morphism between constructible complexes, we let $\pH(f)\colon\pH(\CF)\rightarrow\pH(\CG)$ be the corresponding morphism between cohomologically graded perverse sheaves.
 \item For an algebraic stack $\mathcal{X}$, $\mathsf{MMHM}(\mathcal{X})$ denotes the category of monodromic mixed Hodge modules on $\mathcal{X}$, see \cite[\S 2]{davison2020cohomological}, \cite[\S 2.9]{brav2015symmetries}, \cite{tubach2024mixed} or \cite[\S 5.2]{bu2025cohomology} for details.
 \item The cotangent complex of a derived stack $\CX$ is denoted by $\mathbb{L}_{\CX}$.
\end{enumerate}

\begin{ACK}
    We are grateful to Sebastian Schlegel Mejia for helpful discussions and to Andrei Negu\c t for asking about the torus equivariant version of integrality. The first author is thankful to Shrawan Kumar for pushing him to think about the results in the paper.
    The second author is grateful to Pavel Safronov for discussions on Langlands duality for $3$-manifolds and Yalong Cao for a discussion on Kirwan surjectivity.
	The second author was supported by JSPS KAKENHI Grant Number 25K17229.
\end{ACK}

\section{Cohomological PBW map}

In this section, we define the natural map from left to right in \eqref{eq-intro-main} of \Cref{theorem:main_theorem}.

\subsection{Component lattices and faces}\label{ssec-component-lattice}

Here, we briefly recall the theory of component lattices from \cite{bu2025intrinsic}.
We refer the readers to \cite[\S 2]{bu2025cohomology} for a summary.

Let $\mathcal{X}$ be locally finitely presented derived algebraic stack. In \cite[\S 3.2.2]{bu2025intrinsic}, the authors defined the \textit{component lattice} of $\mathcal{X}$, a presheaf on the category of finite rank free $\mathbb{Z}$-modules. It is defined as
 \[
  \mathrm{CL}(\mathcal{X}) \colon \mathbb{Z}^n \mapsto \pi_0(\mathrm{Map}(\mathrm{B} \mathbb{G}_{\mathrm{m}}^n, \mathcal{X})).   
 \]
 For convenience, we will mostly work with its rationalization $\mathrm{CL}_{\mathbb{Q}}(\mathcal{X}) \coloneqq \mathrm{CL}(\mathcal{X}) \otimes \mathbb{Q}$.
 The category of faces $\mathsf{Face}(\mathcal{X})$ for $\mathcal{X}$ is defined as follows:
 \begin{itemize}
    \item An object is a pair $(F, \alpha)$ of a finite-dimensional $\mathbb{Q}$-vector space $F$ and a morphism $\alpha \colon F \to \mathrm{CL}_{\mathbb{Q}}(\mathcal{X})$.
    \item A morphism $(F', \alpha') \to (F, \alpha)$ is a morphism $f \colon F' \to F$ of $\mathbb{Q}$-vector spaces satisfying $\alpha' = \alpha \circ f$.
 \end{itemize}
For a face $(F, \alpha) \in \Face(\mathcal{X})$, we define the stack $\mathcal{X}_{\alpha}$ as follows: 
First, we take an integral lift $(\mathbb{Z}^n \cong F_{\mathbb{Z}}, \alpha_{\mathbb{Z}}) \in \mathrm{CL}(\mathcal{X})$ of $(F, \alpha)$.
Then we define $\mathcal{X}_{\alpha}$ as the connected component of $\mathrm{Map}(\mathrm{B} \mathbb{G}_{\mathrm{m}}^n, \mathcal{X})$ corresponding to $\alpha_{\mathbb{Z}}$. We can show that $\mathcal{X}_{\alpha}$ does not depend on the choice of the integral lift and is contravariantly functorial in $\alpha$.
In particular, the automorphism group $\Aut(\alpha)$ acts on $\mathcal{X}_{\alpha}$ with contravariance. If $\mathcal{X}$ admits a good moduli space $p \colon \mathcal{X} \to X$, one can show that $\mathcal{X}_{\alpha}$ also admits a good moduli space $p_{\alpha}  \colon \mathcal{X}_{\alpha} \to X_{\alpha}$: see \cite[\S 2.2.6]{bu2025cohomology}.

 A face $(F, \alpha)$ is said to be \textit{non-degenerate} if $\alpha$ does not factor through a lower rank vector space.
 A non-degenerate face $(F, \alpha)$ is said to be \textit{special} if for any non-isomorphic morphism $f \colon \alpha \to \beta$ to a non-degenerate face $\beta$, the map $\mathcal{X}_{\beta} \to \mathcal{X}_{\alpha}$ is non-isomorphic.
 We define the full subcategories 
 \[
    \mathsf{Face}^{\nd}(\mathcal{X}), \mathsf{Face}^{\sp}(\mathcal{X}) \subset \mathsf{Face}(\mathcal{X})
\]
of non-degenerate and special faces.

 For a non-degenerate face $(F, \alpha) \in \mathsf{Face}^{\nd}(\mathcal{X})$, the weights of the cotangent complex of $\CX$ define a hyperplane arrangement on $F$, which we refer to as the cotangent arrangement.
 For a chamber $\sigma \subset F$ with respect to the cotangent arrangement, one can define the stack $\mathcal{X}_{\sigma}^+$ and the correspondence
 \[
    \mathcal{X}_{\alpha}  \xleftarrow[]{\gr_{\sigma}} \mathcal{X}_{\sigma}^+ \xrightarrow[]{\ev_{\sigma}} \mathcal{X}.
 \]
 The stack $\mathcal{X}_{\sigma}^+ $ is defined as a connected component of $\mathrm{Map}(R_{\sigma_{\mathbb{Z}}} / \mathbb{G}_{\mathrm{m}}^{\dim F}, \mathcal{X})$ where $R_{\sigma_{\mathbb{Z}}}$ denotes the affine toric variety associated with an integral lift $\sigma_{\mathbb{Z}}$ of $\sigma$; see \cite[\S 5.1.6]{bu2025intrinsic} for details.

 \begin{example}
    We now explain that the above notions generalize fundamental concepts appearing in Lie theory. We take the quotient stack $\mathcal{X} = V / G$ of a representation $V$ of a connected reductive group $G$.
    Let $\Lambda_T$ denote the cocharacter lattice of the maximal torus $T \subset G$ and $W$ be the Weyl group of $G$. Then we have an equivalence of presheaves on the category of $\mathbb{Q}$-vector spaces
    \[
     \mathrm{CL}_{\mathbb{Q}}(\mathcal{X}) \cong (\Lambda_T \otimes \mathbb{Q}) / W.    
    \]
    A non-degenerate face corresponds to a linear subspace in $\Lambda_T\otimes\BQ$, up to the Weyl group action,
    and special faces are those faces written as intersections of hyperplanes determined by roots of $G$ and $T$-weights of $V$.
    For a non-degenerate face $(F, \alpha)$ and a chamber $\sigma \subset F$ with respect to the cotangent arrangement, pick a cocharacter $\lambda \colon \mathbb{G}_{\mathrm{m}} \to T$ contained in the image of the interior of $\sigma$.
    Then we have
    \[
    \mathcal{X}_{\alpha} \cong V^{\lambda} / L_{\lambda}, \quad \mathcal{X}_{\sigma}^+ \cong V^{\lambda, +} / P_{\lambda}
    \]
    where $V^{\lambda}$ denotes the fixed locus, $V^{\lambda, +}$ the attractor locus, $L_{\lambda}$ the Levi subgroup (centralizer of $\lambda$) and $P_{\lambda}$ the parabolic subgroup, respectively.
 \end{example}

 Now, assume that $\mathcal{X}$ is smooth and almost symmetric, or $0$-shifted symplectic.
 For a non-degenerate face $(F, \alpha) \in \Face^{\nd}(\mathcal{X})$, we define the \textit{cotangent sign representation} of $\Aut(\alpha)$ as follows:
 Take a chamber $\sigma \subset F$ with respect to the cotangent arrangement. For $g \in \Aut(\alpha)$, we define 
 \[
  \sgn_{\alpha}(g) \coloneqq (-1)^{d(\sigma, g(\sigma))}   
 \]
 where $d(\sigma, g(\sigma))$ denotes the signed count of the number of walls between $\sigma$ and $g(\sigma)$; see \cite[\S 4.4.5]{bu2025cohomology} for the details.

 \subsection{BPS cohomology}

 Let $\mathcal{U}$ be a smooth algebraic stack satisfying the assumptions in \Cref{para:assumptions-stacks}. Here, we will introduce the BPS sheaf following \cite[\S 7.1]{bu2025cohomology}.
 Let $(F, \alpha) \in \Face^{\nd}(\mathcal{U})$ be a non-degenerate face and let $p_{\alpha} \colon \mathcal{U}_{\alpha} \to U_{\alpha}$ be the good moduli space morphism.
 We define the BPS sheaf as
 \[
  \mathcal{BPS}^{\alpha}_{U} \coloneqq {}^{\mathrm{p}}\! \mathcal{H}^0(p_{\alpha, *} \mathcal{IC}_{\CU_{\alpha}} \otimes \mathscr{L}^{- \dim F /2}) \in \mathsf{MMHM}(U_{\alpha}).
 \]
 We have the following statement:

 \begin{proposition}
    \label{proposition:smallness_gms}
    We adopt the notations from the last paragraph.
    \begin{enumerate}
     \item The monodromic mixed Hodge module complex $p_{\alpha, *} \mathcal{IC}_{\mathcal{U}_{\alpha}}$ is pure.
     \item For $i< \dim F$, we have $\pH^i(p_{\alpha,*}\ICS_{\mathcal{U}_{\alpha}})=0$. In particular, $\mathcal{BPS}^{\alpha}_{U} = 0$ for non-special $\alpha$.
     \item If the neutral component of the stabilizer of a general closed point of $\mathcal{U}_{\alpha}$ is isomorphic to $\BG_{\rmm}^{\dim F}$, then $\BPS_{U}^{\alpha}=\ICS_{U_{\alpha}}$ and $\BPS_{U}^{\alpha}=0$ otherwise.
    \end{enumerate}
    \end{proposition}
    \begin{proof}
    This is a combination of \cite[Theorem 1.1]{kinjo2024decomposition} and \cite[Proposition~7.2.1 and Corollary~7.2.4]{bu2025cohomology}. 
\end{proof}

\Cref{proposition:smallness_gms} shows that the BPS sheaf is the possibly non-vanishing lowest perverse cohomology. In particular, we have a natural adjunction map
\begin{equation}\label{eq-BPS-map}
    \mathcal{BPS}^{\alpha}_{U} \otimes \mathscr{L}^{\dim F / 2} \to p_{\alpha, *} \mathcal{IC}_{\CU_{\alpha}}.
\end{equation}

\subsection{Global equivariant parameter}\label{ssec-global-equiv-parameter}

Here, we briefly recall the definition of the global equivariant parameter following \cite[\S9.1.2]{bu2025cohomology}.
Let $(F, \alpha) \in \Face^{\nd}(\mathcal{X})$ be a non-degenerate $n$-dimensional face of an algebraic stack.
A \textit{global equivariant parameter} for $\mathcal{X}_{\alpha}$ is a choice of line bundles $\mathcal{L}_1, \ldots, \mathcal{L}_n$ on $\mathcal{X}_{\alpha}$, such that,
the first Chern classes restricted to $\mathrm{B} \mathbb{G}_{\mathrm{m}}^n$ form a basis of $\mathrm{H}^2(\mathrm{B} \mathbb{G}_{\mathrm{m}}^n)$.
The global equivariant parameter exists for a broad class of algebraic stacks: for example, it is shown in \cite[Lemma 9.1.3]{bu2025cohomology} that a quotient stack by an affine algebraic group always admits a global equivariant parameter.

Now let $\mathcal{U}$ and $(F, \alpha)$ as in the last paragraph.
Then, using the global equivariant parameter, one may extend the map \eqref{eq-BPS-map} to
\begin{equation}\label{eq-BPS-gloeq-map}
    \mathcal{BPS}^{\alpha}_{U} \otimes \mathrm{H}^*(\mathrm{B} \mathbb{G}_{\mathrm{m}}^{\dim F})_{\vir} \to p_{\alpha, *} \mathcal{IC}_{\CU_{\alpha}}
\end{equation}
as in \cite[\S9.1.9]{bu2025cohomology}.

\subsection{Cohomological Hall induction}\label{ssec-cohi}

Let $\mathcal{U}$ be a smooth algebraic stack satisfying the assumptions in \Cref{para:assumptions-stacks}, and take a non-degenerate face $(F, \alpha) \in \Face^{\nd}(\mathcal{X})$ and a chamber $\sigma \subset F$ in the cotangent arrangement.
Consider the following commutative diagram:
\[\begin{tikzcd}
	& {\mathcal{U}_{\sigma}^+} \\
	{\mathcal{U}_{\alpha}} && {\mathcal{U}} \\
	{U_{\alpha}} && {U.}
	\arrow["{\gr_{\sigma}}"', from=1-2, to=2-1]
	\arrow["{\ev_{\sigma}}", from=1-2, to=2-3]
	\arrow["{p_{\alpha}}"', from=2-1, to=3-1]
	\arrow["p", from=2-3, to=3-3]
	\arrow["g_{\alpha}", from=3-1, to=3-3]
\end{tikzcd}\]
Since $\mathcal{U}$ is smooth, $\gr_{\sigma}$ is also smooth, and the existence of the good moduli space implies that $\ev_{\sigma}$ is proper.
It is also shown in \cite[\S 8.1.3]{bu2025intrinsic} that $g_{\alpha}$ is a finite morphism.

Using the smoothness of $\gr_{\sigma}$ and the properness of $\ev_{\sigma}$, we have the following natural map
\begin{equation}\label{eq-rel-CoHI}
 *^{\mathcal{H}\mathrm{all}}_{\sigma} \colon g_{\alpha, *} p_{\alpha, *} \mathcal{IC}_{\mathcal{U}_{\alpha}} \to  p_* \mathcal{IC}_{\mathcal{U}}
\end{equation}
which we refer to as the \textit{relative cohomological Hall induction} (CoHI) map. By passing to the global sections, we obtain the \textit{absolute cohomological Hall induction} map
\[
    *^{\mathrm{Hall}}_{\sigma} \colon \mathrm{H}^*(\mathcal{U}_{\alpha})_{\mathrm{vir}} \to \mathrm{H}^*(\mathcal{U})_{\mathrm{vir}}.
\]
Since $\CU$ is almost symmetric, the CoHI map preserves the cohomological gradings. Assume further that we are given a non-degenerate face $(F', \alpha') \in \Face^{\mathrm{nd}}(\mathcal{U}_{\alpha})$ and a cone $\sigma' \subset F'$.
In this case, we can define the chamber $\sigma \uparrow \sigma' \subset F'$ with respect to the cotangent arrangement for $\mathcal{U}$, together with the following commutative diagram
\[\begin{tikzcd}
	&& {\mathcal{U}_{\sigma \uparrow \sigma'}} \\
	& {(\mathcal{U}_{\alpha})_{\sigma'}} && {\mathcal{U}_{\sigma}} \\
	{(\mathcal{U}_{\alpha})_{\alpha'}} && {\mathcal{U}_{\alpha}} && {\mathcal{U}.}
	\arrow[from=1-3, to=2-2]
	\arrow[from=1-3, to=2-4]
	\arrow["{\gr_{\sigma \uparrow \sigma'}}"', curve={height=30pt}, from=1-3, to=3-1]
	\arrow["{\ev_{\sigma \uparrow \sigma'}}", curve={height=-30pt}, from=1-3, to=3-5]
	\arrow["{\gr_{\sigma'}}", from=2-2, to=3-1]
	\arrow["{\ev_{\sigma'}}"', from=2-2, to=3-3]
	\arrow["{\gr_{\sigma}}", from=2-4, to=3-3]
	\arrow["{\ev_{\sigma}}"', from=2-4, to=3-5]
\end{tikzcd}\]
such that the middle square is Cartesian; see \cite[Theorem 6.3.6]{bu2025intrinsic}.
In particular, we have the associativity relation of the cohomological Hall induction
\begin{equation}\label{eq-CoHI-assoc}
 *_{\sigma}^{\mathcal{H}\mathrm{all}} \circ *_{\sigma'}^{\mathcal{H}\mathrm{all}} =     *_{\sigma \uparrow \sigma'}^{\mathcal{H}\mathrm{all}}.
\end{equation}

\subsection{Cohomological PBW map}

Now let $\mathcal{U}$, $(F, \alpha) \in \mathsf{Face}^{\nd}(\mathcal{X})$ and $\sigma \subset F$ as in the previous paragraph.
We also fix a global equivariant parameter for $\mathcal{U}_{\alpha}$. By combining \eqref{eq-BPS-gloeq-map} and \eqref{eq-rel-CoHI}, we obtain a map
\[
   \Gamma_{\sigma} \colon g_{\alpha, *} \mathcal{BPS}^{\alpha}_{U} \otimes \mathrm{H}^*(\mathrm{B} \mathbb{G}_{\mathrm{m}}^{\dim F})_{\mathrm{vir}} \to p_* \mathcal{IC}_{\mathcal{U}}.
\]
We now choose a chamber $\sigma_{\alpha} \subset F$ with respect to the cotangent arrangement for each $(F, \alpha)$ and define the \textit{relative cohomological PBW map} as
\begin{equation}\label{eq-cohint-map}
 \Psi = \Psi_{\mathcal{U}} \colon \bigoplus_{(F, \alpha) \in \Face^{\nd}(\mathcal{U})} (g_{\alpha, *} \mathcal{BPS}^{\alpha}_{U} \otimes \mathrm{H}^*(\mathrm{B} \mathbb{G}_{\mathrm{m}}^{\dim F})_{\mathrm{vir}} \otimes \sgn_{\alpha})^{\mathrm{Aut}(\alpha)} \xrightarrow{ \sum_{\alpha} \Gamma_{\sigma_{\alpha}}} p_* \mathcal{IC}_{\mathcal{U}}.
\end{equation}
The sum runs over the set of isomorphism classes of non-degenerate faces of $\CU$. For a smooth algebraic stack $\mathcal{U}$ satisfying the assumptions in \Cref{para:assumptions-stacks}, we say that $\mathcal{U}$ satisfies the \textit{BPS decomposition theorem}
if $ \Psi $ is an isomorphism for any choice of the global equivariant parameter and of chambers $\sigma_{\alpha} \subset F$.
By \Cref{proposition:smallness_gms}, only special faces contribute to the cohomological PBW map.

Our main theorem \ref{theorem:main_theorem} can be rephrased as follows:

\begin{theorem}\label{thm-cohint-satisfy}
    Any smooth algebraic stack $\mathcal{U}$ satisfying the assumptions in \Cref{para:assumptions-stacks} satisfies the BPS decomposition theorem.
\end{theorem}

We have the following special case, which will be the key initial step in the induction.

\begin{proposition}\label{prop-cohint-BG}
   The BPS decomposition theorem holds for $\mathrm{B} G$ for a reductive group $G$.
\end{proposition}

\begin{proof}
    This is a special case of \cite[Theorem 9.1.12]{bu2025intrinsic}.
\end{proof}

The following statement is essentially proved in \cite[\S 9.2]{bu2025cohomology}:

\begin{proposition}\label{prop-slice-reduction}
    To prove \Cref{thm-cohint-satisfy}, it is enough to show that the BPS decomposition theorem holds for quotient stacks $\mathcal{U} = V / G$,
    where $G$ is a connected reductive group and $V$ is a representation of $G$, which does not contain the trivial representation as a direct summand and 
    there is no non-trivial subtorus of $\mathrm{Z}(G)$ which acts trivially on $V$.
\end{proposition}

\begin{proof}
    The reduction to quotient stacks $\mathcal{U} = V / G$ for possibly disconnected $G$ and any representation $V$ of $G$ is a consequence of the \'etale slice theorem due to \textcite[Theorem 4.12]{alper2020luna} and \cite[\S 9.2.4]{bu2025cohomology}.
    By the argument in \cite[\S 9.2.3]{bu2025cohomology}, we may assume that $G$ is connected, and the argument in \cite[\S 9.2.2]{bu2025cohomology} shows that we may further assume that $V$ does not contain a trivial representation as a direct summand.
    Now assume that there is a subtorus $T' \subset \mathrm{Z}(G)$ such that the action of $T'$ on $V$ is trivial. We set $\bar{G} \coloneqq G /T'$ and $\bar{\mathcal{U}} \coloneqq V / T'$.
    Then the composition
    \[
     \mathsf{Face}^{\mathrm{nd}}(\mathcal{U}) \to \mathsf{Face}(\bar{\mathcal{U}}) \xrightarrow{(-)^{\mathrm{nd}}} \mathsf{Face}^{\mathrm{nd}}(\bar{\mathcal{U}}), \quad \alpha \mapsto \bar{\alpha}
    \]
    gives an identification of special faces for $\mathcal{U}$ and $\bar{\mathcal{U}}$, where the latter functor is the non-degenerate quotient functor \cite[Definition 3.3.2]{bu2025intrinsic}.
    By construction, we have $U_{\alpha} \cong U_{\bar{\alpha}}$ and the BPS sheaves are identified for special faces: $\BPS_{\CU}^{\alpha}\cong\BPS_{\overline{\CU}}^{\overline{\alpha}}$.
    In particular, we have 
    \[
     \Psi_{\mathcal{U}} =  \Psi_{\bar{\mathcal{U}}} \otimes \id_{\mathrm{H}^* (\mathrm{B} T')_{\vir}}.
    \]
    Therefore we conclude that the BPS decomposition theorem for $\bar{\mathcal{U}}$ implies the BPS decomposition theorem for $\mathcal{U}$ as desired. 
\end{proof}

For later use, we record the following statement:

\begin{proposition}\label{prop-crit-reduction}
    Let $G$ be a connected reductive group and $V$ be a representation of $G$. 
    We let $\mathfrak{g}_{\mathrm{ad}}$ denote the Lie algebra of the adjoint group $G_{\mathrm{ad}}$ regarded as a representation of $G$.
    Then the BPS decomposition theorem for $(V \times \mathfrak{g}_{\mathrm{ad}}) / G$ implies the BPS decomposition theorem for $V / G$.
\end{proposition}

\begin{proof}
    This is essentially proved in \cite[\S 9.3.5]{bu2025cohomology}.
    Namely, for a smooth Artin stack $\mathcal{U}$ and a function $f$ on it, it is proved in \textit{loc.cit.} that the BPS decomposition theorem for $\mathcal{U}$ implies the BPS decomposition theorem for the critical locus $\mathrm{Crit}(f)$.
    Now let $q$ be the function on $(V \times \mathfrak{g}_{\mathrm{ad}}) / G$ induced by the Killing form on $\mathfrak{g}_{\mathrm{ad}}$.
    Then we have $\mathrm{Crit}(q) = V / G$, hence we obtain the desired result.
\end{proof}

\section{Proof of the BPS decomposition theorem}\label{sec-proof-main}

We now prove \Cref{thm-cohint-satisfy}. 
By \Cref{prop-slice-reduction}, we may assume $\mathcal{U}$ is a quotient stack of the form $V / G$, where $G$ is a connected reductive group, $V$ is a representation of $G$ such that $V$ does not contain the trivial representation as a direct summand and there is no non-trivial subtorus of $\mathrm{Z}(G)$ which acts trivially on $V$.
A key ingredient of the proof is the following cohomological bound, refining the bound obtained by the first author in \cite[Proposition 1.4]{hennecart2024cohomological}.
 
\begin{theorem}\label{thm-cohomological-bound}
    Let $G$ and $V$ be as above, and set $\tilde{\mathcal{U}} = (V \times \mathfrak{g}_{\mathrm{ad}}) / G$. Consider the following map 
    \begin{equation}\label{eq-Phi'}
     \Phi_{\tilde{\mathcal{U}}}' \colon \bigoplus_{\substack{(F, \alpha) \in \Face^{\nd}(\tilde{\mathcal{U}}) \\ F \neq 0}}\mathrm{H}^*(\tilde{\mathcal{U}}_{\alpha})_{\mathrm{vir}} \xrightarrow{\sum_{(F, \alpha)}*^{\mathrm{Hall}}_{\sigma}} \mathrm{H}^*(\tilde{\mathcal{U}})_{\mathrm{vir}} 
    \end{equation}
    and let $P_{\tilde{\mathcal{U}}}$ be the cokernel of $\Phi'_{\tilde{\CU}}$. Then $P_{\tilde{\mathcal{U}}}$ is concentrated in non-positive cohomological degrees.
 \end{theorem}

 This will be proven in the next section following \cite[Proposition 1.4]{hennecart2024cohomological} and we will give another proof in \Cref{section-geometric-proof} based on a geometric argument.
We now start the proof of \Cref{thm-cohint-satisfy} assuming \Cref{thm-cohomological-bound} above.
 We will proceed by the induction both on the dimension of the group and the rank of the representation.
 When $V = 0$, the statement is already proved in \Cref{prop-cohint-BG}, and when $G = \{ e \}$, the theorem is tautological, since the good moduli space morphism is then an isomorphism. Therefore we may assume that the BPS decomposition theorem holds for $W /H$ if $\dim W < \dim V$ or $\dim H < \dim G$.
 Also, recall that it is enough to prove the BPS decomposition theorem for $\tilde{\mathcal{U}} = (V \times \mathfrak{g}_{\mathrm{ad}}) / G$ by \Cref{prop-crit-reduction}.

 First, by using the induction hypothesis on the dimension of the group and the \'etale slice theorem,
 we see that the map $\Psi_{\tilde{\mathcal{U}}}$ is an isomorphism outside of the origin $0 \in \tilde{U} = (V \times \mathfrak{g}_{\mathrm{ad}}) \GIT G$.
 In particular, to prove that the map $\Psi_{\tilde{\mathcal{U}}}$ is an isomorphism, it is enough to show that the induced map on the global sections $\mathrm{H}^*(\Psi_{\tilde{\mathcal{U}}})$ is an isomorphism.

 We now claim that, for each face $(F, \alpha) \in \Face^{\mathrm{nd}}(\tilde{\mathcal{U}})$ with $F \neq 0$, the BPS decomposition theorem holds for $\tilde{\mathcal{U}}_{\alpha}$.
 We let $\lambda_F$ denote a generic cocharacter contained in $F$.
 Then we have $\tilde{\mathcal{U}}_{\alpha} = (V \times \mathfrak{g}_{\mathrm{ad}})^{\lambda_F}  / L_{\lambda_F}$ and $\mathcal{U}_{\bar{\alpha}} = V^{\lambda_F}  / L_{\lambda_F}$ where $\bar{\alpha} \in\Face^{\nd}(\CU)$ denotes the image of $\alpha$ by the map $\Face^{\nd}(\tilde{\CU}) \to \Face^{\nd}(\CU)$. 
 By the induction hypothesis, the BPS decomposition theorem holds for $\mathcal{U}_{\bar{\alpha}}$, namely, the map $\Psi_{\mathcal{U}_{\bar{\alpha}}}$ is an isomorphism.
 Now, by using the induction hypothesis again, 
 the BPS decomposition theorem holds for $\tilde{\mathcal{U}}_{\alpha} = (V \times \mathfrak{g}_{\mathrm{ad}})^{\lambda_F}  / L_{\lambda_F}$ outside the origin, namely, the map $\Psi_{\tilde{\mathcal{U}}_{\alpha}}$ has cofibre $K$ supported on the origin.
 Let $q_{\alpha} \colon \tilde{\mathcal{U}}_{\alpha} = (V \times \mathfrak{g}_{\mathrm{ad}})^{\lambda_F}  / L_{\lambda_F} \to \mathbb{A}^1$ be a non-degenerate quadratic function on $\mathfrak{g}_{\mathrm{ad}}^{\lambda_F}\cong \mathrm{Z}(\mathfrak{g}_{\mathrm{ad}}^{\lambda_F})\oplus [\mathfrak{g}_{\mathrm{ad}}^{\lambda_F},\mathfrak{g}_{\mathrm{ad}}^{\lambda_F}]$ induced by the Killing form of $[\mathfrak{g}_{\mathrm{ad}}^{\lambda_F},\mathfrak{g}_{\mathrm{ad}}^{\lambda_F}]$ and any non-degenerate quadratic function on $\mathrm{Z}(\mathfrak{g}_{\mathrm{ad}}^{\lambda_F})$.
 The critical locus of $q_{\alpha}$ is $\mathcal{U}_{\bar{\alpha}}$, and if we write $\bar{q}_{\alpha} \colon \tilde{U}_{\alpha} \to \mathbb{A}^1$ the function induced by $q_{\alpha}$,
 the map $\varphi_{\bar{q}_{\alpha}}(\Psi_{\tilde{\mathcal{U}}_{\alpha}})$ has cofibre isomorphic to $K$ since $K$ is supported on the zero-locus of $\overline{q}_{\alpha}$.
 Since $\varphi_{\bar{q}_{\alpha}}(\Psi_{\tilde{\mathcal{U}}_{\alpha}})$ is naturally identified with $\Psi_{\mathcal{U}_{\bar{\alpha}}}$ (see \cite[\S 9.2.3]{bu2025cohomology} for the detail), we conclude that $K = 0$ as desired.

 We now consider the following partial cohomological PBW maps
 \begin{align*}
    \Psi_{\tilde{\mathcal{U}}}' \colon &\bigoplus_{\substack{(F, \alpha) \in \Face^{\nd}(\mathcal{U}) \\ F \neq 0}} (g_{\alpha, *} \mathcal{BPS}^{\alpha}_{\tilde{U}} \otimes \mathrm{H}^*(\mathrm{B} \mathbb{G}_{\mathrm{m}}^{\dim F})_{\mathrm{vir}} \otimes \sgn_{\alpha})^{\mathrm{Aut}(\alpha)} \xrightarrow{ \sum_{\alpha} \Gamma_{\sigma_{\alpha}}} p_* \mathcal{IC}_{\tilde{\mathcal{U}}} \\
    \Psi_{\tilde{\mathcal{U}}}'' \colon &\bigoplus_{\substack{(F, \alpha) \in \Face^{\nd}(\mathcal{U}) \\ F \neq 0, \Lambda_T\otimes\BQ}} (g_{\alpha, *}  \mathcal{BPS}^{\alpha}_{\tilde{U}} \otimes \mathrm{H}^*(\mathrm{B} \mathbb{G}_{\mathrm{m}}^{\dim F})_{\mathrm{vir}} \otimes \sgn_{\alpha})^{\mathrm{Aut}(\alpha)} \xrightarrow{ \sum_{\alpha} \Gamma_{\sigma_{\alpha}}} p_* \mathcal{IC}_{\tilde{\mathcal{U}}}.
 \end{align*}
 We first claim that the map ${}^\mathrm{p}\! \mathcal{H}( \Psi_{\tilde{\mathcal{U}}}'' )$ induced on the perverse cohomologies is injective.
 To see this, since we already know that $\Psi_{\tilde{\mathcal{U}}}$ is an isomorphism outside the origin, it is enough to show that the kernel of ${}^\mathrm{p}\! \mathcal{H}( \Psi_{\tilde{\mathcal{U}}}'' )$ does not contain any shifted constant sheaf supported on the origin.
 However, since $g_{\alpha}$ is finite and $\mathcal{BPS}^{\alpha}_{\tilde{U}}$ is either zero or the intersection complex, we conclude that the source of ${}^\mathrm{p}\! \mathcal{H}( \Psi_{\tilde{\mathcal{U}}}'' )$ does not contain a constant sheaf supported on the origin as a direct summand.
In particular, we conclude that ${}^\mathrm{p}\! \mathcal{H}( \Psi_{\tilde{\mathcal{U}}}'' )$  induces an injection, hence so does the global cohomology $\mathrm{H}^*(\Psi_{\tilde{\mathcal{U}}}'' )$ by \Cref{lem-injective} (1).
The cofibre of ${}^\mathrm{p}\! \mathcal{H}( \Psi_{\tilde{\mathcal{U}}}'' )$ is of the form $L \oplus \mathcal{BPS}_{\tilde{U}}^{\alpha_{\mathrm{min}}}$ where $\alpha_{\mathrm{min}}$ is the minimal face (i.e. the unique face with $F = 0$) and $L$ is supported on the origin.
Further, since $L$ is a direct summand of $p_{*} \mathcal{IC}_{\tilde{\mathcal{U}}}$, by \Cref{proposition:smallness_gms}, $L$ is concentrated in positive degrees (since we assume $V \neq 0$). The morphism $\BPS_{\tilde{U}}^{\alpha_{\mathrm{min}}}\rightarrow \cofib(\Psi''_{\mathcal{U}})$ has cofibre $L$ and satisfies the assumptions of \Cref{lem-injective}. By passing to the global sections and using \Cref{lem-injective} (2), we obtain a natural short exact sequence
\begin{equation}\label{eq-compare-1}
 0 \to     \mathrm{H}^*(\mathcal{BPS}_{\tilde{U}}^{\alpha_{\mathrm{min}}}) \to \mathrm{coker} (\mathrm{H}^*(\Psi''_{\tilde{\CU}})) \to \mathrm{H}^*(L) \to 0.
\end{equation}

We next prove the following equality
\[
 \Im \mathrm{H}^*(\Psi_{\tilde{\mathcal{U}}}' ) = \Im\Phi_{\tilde{\mathcal{U}}}' .
\]
Here, the map $\Phi'_{\tilde{\mathcal{U}}}$ is defined in \eqref{eq-Phi'}. The inclusion $\subset$ is obvious.
To prove the reverse inclusion, take an element $a \in \mathrm{H}^*(\tilde{\mathcal{U}}_{\alpha})$ for some $ (F, \alpha) \in \Face^{\mathrm{nd}}(\tilde{\mathcal{U}})$ with $F \neq 0$.
Since the cohomological Hall induction map $\mathrm{H}^*(\Gamma_{\sigma_{\alpha}})$ is invariant under the $\Aut(\alpha)$-action as shown in \cite[(8.2.2.1)]{bu2025cohomology},
the image of $\mathrm{H}^*(\Psi_{\tilde{\mathcal{U}}}' )$ is unchanged if we do not take $\Aut(\alpha)$-invariant part in the source.
In particular, using the associativity of the cohomological Hall induction \eqref{eq-CoHI-assoc}, it is enough to show that $a$ is contained in the image of $\mathrm{H}^*(\Psi_{\tilde{\mathcal{U}}_{\alpha}} )$.
However, this comes from the fact that $\tilde{\mathcal{U}}_{\mathcal{\alpha}}$ satisfies the BPS decomposition theorem, as we have already shown.
In particular, we have a short exact sequence
\begin{equation}\label{eq-compare-2}
   (\mathrm{H}^*(\mathrm{B} T)_{\mathrm{vir}} \otimes \mathrm{sgn}_{\alpha_{\mathrm{max}}})^{\Aut(\alpha_{\mathrm{max}})} \to \mathrm{coker} (\mathrm{H}^*(\Psi_{\tilde{\mathcal{U}}}''))  \to P_{\tilde{\mathcal{U}}} \to 0,
\end{equation}
where the first map is induced by the cohomological Hall induction.
Here, $\alpha_{\mathrm{max}}$ denotes the maximal non-degenerate face, hence $\Aut(\alpha_{\mathrm{max}})$ is the Weyl group $W$ of $G$.
We now look at \eqref{eq-compare-1} and \eqref{eq-compare-2}.
Note that $\mathrm{H}^*(\mathcal{BPS}_{\tilde{U}}^{\alpha_{\mathrm{min}}})$ and $P_{\tilde{\mathcal{U}}}$ are concentrated in non-positive cohomological degree by the Artin vanishing and \Cref{thm-cohomological-bound}.
On the other hand, $\mathrm{H}^*(L)$ and $ (\mathrm{H}^*(\mathrm{B} T)_{\mathrm{vir}} \otimes \mathrm{sgn}_{\alpha_{\mathrm{max}}})^{\Aut(\alpha_{\mathrm{max}})}$ are concentrated in positive cohomological degrees (as $T \neq 0$).
In particular, the composition
\[
    \mathrm{H}^*(\mathcal{BPS}_{\tilde{U}}^{\alpha_{\mathrm{min}}}) \to  \mathrm{coker} (\mathrm{H}^*(\Psi_{\tilde{\mathcal{U}}}''))  \to P_{\tilde{\mathcal{U}}} 
\]
is an isomorphism. Together with \eqref{eq-compare-2}, this shows that the absolute cohomological PBW map $\mathrm{H}^*(\Psi_{\tilde{\mathcal{U}}})$ is surjective. Since the relative cohomological PBW map $\Psi_{\tilde{\mathcal{U}}}$ is an isomorphism outside the origin by the induction hypothesis, we see that $\pH(\Psi_{\tilde{\mathcal{U}}})$ is also surjective by \Cref{lem-surjective}. We now claim that $\pH(\Psi_{\tilde{\mathcal{U}}})$ is injective. To see this, since we already know that it is an isomorphism outside the origin and that $\mathcal{BPS}_{\tilde{U}}^{\alpha}$ has zero-dimensional support if and only if $\alpha = \alpha_{\mathrm{max}}$, it is enough to show that the cohomological Hall induction map between the perverse cohomology
\[
    (g_{\alpha_{\mathrm{max}}, *} \mathcal{BPS}^{\alpha_{\mathrm{max}}}_{\tilde{U}} \otimes \mathrm{H}^*(\mathrm{B} T)_{\mathrm{vir}} \otimes \sgn_{\alpha_{\mathrm{max}}})^{\mathrm{Aut}(\alpha_{\mathrm{max}})} \xrightarrow{\pH(*_{\sigma_{\alpha_{\mathrm{max}}}}^{\mathcal{H}\mathrm{all}})} \pH (p_* \mathcal{IC}_{\tilde{\mathcal{U}}})
\]
is injective. Since $\mathcal{BPS}^{\alpha_{\mathrm{max}}}_{\tilde{U}}\cong\BQ_{\{0\}}$ has zero-dimensional support, the injectivity is equivalent to the injectivity on the global section
\[
     \Upsilon \colon (\mathrm{H}^*(\mathrm{B} T)_{\mathrm{vir}} \otimes \mathrm{sgn}_{\alpha_{\mathrm{max}}})^{\Aut(\alpha_{\mathrm{max}})} \to \gr_{\mathfrak{P}}\mathrm{H}^*(\tilde{\mathcal{U}})_{\vir},
\]
where $\gr_{\mathfrak{P}}$ denote the associated graded with respect to the perverse filtration.

We will deduce the contradiction by computing the asymptotic behaviour of the graded dimension.
For a $\mathbb{Z}$-graded vector space $V^{*}$, we let $P_t(V^*)$ denote the function 
\[
    n \mapsto \dim V^0 + \dim V^1 +  \cdots + \dim V^{n}.
\] 
First, we claim the inequality 
\begin{equation}\label{eq-easy-estimate}
    \liminf_{t \to \infty} P_t ((\mathrm{H}^*(\mathrm{B} T)_{\mathrm{vir}} \otimes \mathrm{sgn}_{\alpha_{\mathrm{min}}})^{\Aut(\alpha_{\mathrm{min}})}) / t^{\dim T} > 0.
\end{equation}
To see this, note first that 
\begin{equation}\label{eq-U-estimate}
    \liminf_{t \to \infty} P_t (\mathrm{H}^*(\tilde{\mathcal{U}}))_{\mathrm{vir}} / t^{\dim T} > 0,
\end{equation}
which follows, for example, from the fact that $\mathrm{H}^*(\mathrm{B} G)$ is isomorphic to the polynomial ring with $\dim T$-variables. On the other hand, for $\alpha \neq \alpha_{\mathrm{max}}$, we have
\begin{equation}\label{eq-obvious-estimate}
    \limsup_{t \to \infty} P_t((\mathrm{H}^*( \mathcal{BPS}^{\alpha}_{\tilde{U}}) \otimes \mathrm{H}^*(\mathrm{B} \mathbb{G}_{\mathrm{m}}^{\dim F})_{\mathrm{vir}} \otimes \sgn_{\alpha})^{\mathrm{Aut}(\alpha)}) / t^{\dim T } = 0
\end{equation}
since $\mathrm{H}^*( \mathcal{BPS}^{\alpha}_{\tilde{U}})$ is finite-dimensional and $\dim F < \dim T$. In particular, the surjectivity of the cohomological PBW map $\mathrm{H}^*(\Psi_{\tilde{\mathcal{U}}})$ implies the inequality \eqref{eq-easy-estimate}.
Therefore we have $(\mathrm{H}^*(\mathrm{B} T)_{\mathrm{vir}} \otimes \mathrm{sgn}_{\alpha_{\mathrm{max}}})^{\Aut(\alpha_{\mathrm{max}})} \neq 0$.    
Next, we claim that 
\begin{equation}\label{eq-important-estimate}
    \lim_{t \to \infty} \left(P_t ((\mathrm{H}^*(\mathrm{B} T)_{\mathrm{vir}} \otimes \mathrm{sgn}_{\alpha_{\mathrm{max}}})^{\Aut(\alpha_{\mathrm{max}})}) -     P_t (\mathrm{H}^*(\mathrm{B} T)^{\Aut(\alpha_{\mathrm{max}})})\right) / t^{\dim T} = 0.
\end{equation}
To see this, take a non-zero element $e \in (\mathrm{H}^N(\mathrm{B} T) \otimes \mathrm{sgn}_{\alpha_{\mathrm{max}}})^{\Aut(\alpha_{\mathrm{max}})}$ for some $N\in\BN$.
Then we have a sequence of injections
\[
    \mathrm{H}^*(\mathrm{B} T)^{\Aut(\alpha_{\mathrm{max}})} \xrightarrow{\cdot e} (\mathrm{H}^{* + N}(\mathrm{B} T) \otimes \mathrm{sgn}_{\alpha_{\mathrm{max}}})^{\Aut(\alpha_{\mathrm{max}})} \xrightarrow{\cdot e}  \mathrm{H}^{* + 2N}(\mathrm{B} T)^{\Aut(\alpha_{\mathrm{max}})},
\] 
which gives the desired equality.

Assume now that $\Upsilon$ is not injective, and take $0 \neq a \in (\mathrm{H}^*(\mathrm{B} T)_{\mathrm{vir}} \otimes \mathrm{sgn}_{\alpha_{\mathrm{max}}})^{\Aut(\alpha_{\mathrm{max}})}$ with $\Upsilon (a) = 0$.
Since the cohomological Hall induction map is a morphism of modules over $\mathrm{H}^*(\tilde{\mathcal{U}})$,
we see that $\Upsilon$ is a module homomorphism over $\mathrm{H}^*(\mathrm{B} G) \cong \mathrm{H}^*(\mathrm{B} T)^{\mathrm{Aut}(\alpha_{\mathrm{max}})}$.
In particular, for any $c \in \mathrm{H}^*(\mathrm{B} G) \cong \mathrm{H}^*(\mathrm{B} T)^{\mathrm{Aut}(\alpha_{\mathrm{max}})}$, we have $\Upsilon (c \cdot a) = 0$.
Since the multiplication map $\cdot a$ induces an injection,
using \eqref{eq-important-estimate}, we conclude that 
\[
    0\leq\lim_{t \to \infty} P_t (\Im \Upsilon) / t^{\dim T}\leq \lim_{t\to\infty} P_t \left((\mathrm{H}^*(\mathrm{B} T)_{\mathrm{vir}} \otimes \mathrm{sgn}_{\alpha_{\mathrm{max}}})^{\Aut(\alpha_{\mathrm{max}})}/a\cdot\mathrm{H}^*(\mathrm{B} T)^{\mathrm{Aut}(\alpha_{\mathrm{max}})} \right)/t^{\dim T}= 0.
\]
However, this contradicts the fact that $\mathrm{H}^* (\pH(\Psi_{\tilde{\mathcal{U}}}))$ is surjective, since we have \eqref{eq-U-estimate} and \eqref{eq-obvious-estimate}.

This completes the proof of \Cref{thm-cohint-satisfy}.

\section{Cohomological bounds for algebraic BPS cohomology}
\label{section:cohomological_bounds}


We now prove \Cref{thm-cohomological-bound}.
This is a slight refinement of the bounds given by \cite[Proposition~1.4]{hennecart2024cohomological}. For completeness, we recall the proof here, using \cite[Lemma~3.6]{efimov2012cohomological}.

In the notations of \emph{loc.cit.}, we let $B\coloneqq \Sym(\Ft^*)\cong\BQ[z_1,\hdots,z_{\dim \Ft}]\cong\rmH^*_{T}(\pt)$. To match the cohomological grading, we set $\deg(z_i)=2$ (and so degrees of polynomials are doubled compared to \cite{efimov2012cohomological}). We consider the finite collection of polynomials
\[
\left\{\prod_{\substack{\alpha\in\rmX^*(T)\\\langle\lambda,\alpha\rangle<0}}\alpha^{\dim V_{\alpha}}\mid \lambda\in\Lambda_T\text{ such that } V^{\lambda}\times\Fg^{\lambda}_{\mathrm{ad}}\subsetneq V\times \Fg_{\mathrm{ad}} \right\}\in B\,.
\]
We may relabel these polynomials $\{P_1,\hdots,P_r\}$. They have the form described by \cite[Lemma~3.6]{efimov2012cohomological}, where in the notations of \emph{loc.cit.} $\{l_1,\hdots,l_s\}\subset \Ft^*= B^2$ is a collection of linear forms which is a set of representatives of the non-zero weights of $V$ up to non-zero scalars. Up to a non-zero scalar, we may write $P_i=\prod_{j=1}^sl_j^{d_{i,j}}$. Then, $\sum_{j=1}^s\max_{1\leq i\leq r}d_{i,j}=\frac{\dim V}{2}$. Moreover, by \cite[Proposition~4.19]{hennecart2024cohomological}, the ideal $(P_1,\hdots,P_s)\subset B$ has finite codimension. Using the equivalence \cite[Lemma~3.6 (1)$\iff$ (2)]{efimov2012cohomological}, we obtain $B^{2d}\subset J_0$ for $2d\geq \dim V-2\dim (T)+2$, where $J_0\subset B$ is such that $J_0^{W}$ is the image of $\Phi'_{\tilde{\mathcal{U}}}$. By taking into account the cohomological shift by $\dim V - \dim \mathrm{Z}(G^{\circ})$ in the definition of $P^i_{\tilde{\mathcal{U}}}$, we obtain $P^i_{\tilde{\mathcal{U}}}=0$ for $i\geq\dim V-2\dim T+2-\dim V + \dim \mathrm{Z}(G^{\circ})= \dim \mathrm{Z}(G^{\circ}) -2\dim T +2$.
This concludes if $\dim T > 0$, and the statement is obvious if $T = \{ e \}$.

\section{BPS decomposition for $0$-shifted and $(-1)$-shifted symplectic stacks }\label{section:applications}

In this section, we will prove the BPS decomposition theorem for $(-1)$-shifted symplectic stacks and $0$-shifted symplectic stacks as an application of our main result Theorem \ref{theorem:main_theorem}.
Throughout the section, we will use the language of shifted symplectic geometry introduced by Pantev--To\"en--Vaqui\'e--Vezzosi \cite{pantev2013shifted}.
We refer the readers to \cite[\S 3]{bu2025cohomology} for a summary of shifted symplectic geometry we use in this section.

\subsection{BPS decomposition for $(-1)$-shifted symplectic stacks}

In this subsection, we let $\mathcal{X}$ be a $(-1)$-shifted symplectic stack satisfying the assumptions in \Cref{para:assumptions-stacks}.
Assume further that $\mathcal{X}$ is equipped with an orientation, i.e., a choice of a $\mathbb{Z}/ 2 \mathbb{Z}$-graded line bundle $\mathcal{L}$ on $\mathcal{X}$ and an isomorphism $o \colon \mathcal{L}^{\otimes 2} \cong \det(\mathbb{L}_{\mathcal{X}})$.
Under this situation, \textcite{ben2015darboux} constructed a monodromic mixed Hodge module
\[
    \varphi_{\mathcal{X}} =  \varphi_{\mathcal{X}, o} \in \mathsf{MMHM}(\mathcal{X})    
\]
called the \textit{Donaldson--Thomas mixed Hodge module}.
We refer to \cite[\S 6]{bu2025cohomology} for a brief summary of the Donaldson--Thomas mixed Hodge modules.

Let $(F, \alpha) \in \Face^{\nd}(\mathcal{X})$ be a non-degenerate face.
It is shown in \cite[\S 6.1.11]{bu2025cohomology} that $\mathcal{X}_{\alpha}$ is naturally equipped with a $(-1)$-shifted symplectic structure and an $\Aut(\alpha)$-equivariant orientation $\alpha^{\star} o$.
In particular, we can define the Donaldson--Thomas mixed Hodge module $\varphi_{\mathcal{X}_{\alpha}}$ on $\mathcal{X}_{\alpha}$ and it is equivariant with respect to the natural $\Aut(\alpha)$-action.
Let $p_{\alpha} \colon \mathcal{X}_{\alpha} \to X_{\alpha}$ be the good moduli space morphism.
We define the BPS sheaf $\mathcal{BPS}_{X}^{\alpha}$ by
\[
    \mathcal{BPS}_{X}^{\alpha}  \coloneqq \mathcal{H}^0(p_{\alpha, *} \varphi_{\mathcal{X}_{\alpha}} \otimes \mathscr{L}^{- \dim F / 2}) \in \mathsf{MMHM}(X_{\alpha}).
\]
It is equivariant with respect to the $\Aut(\alpha)$-action.
In \cite[Proposition 7.2.9]{bu2025cohomology}, it is shown that the BPS sheaf is the lowest possibly non-vanishing perverse cohomology. Therefore there is a natural morphism
\[
 \mathcal{BPS}^{\alpha}_{X} \otimes \mathscr{L}^{\dim F / 2} \to p_{\alpha, *} \varphi_{\mathcal{X}_{\alpha}}.
\]
Using the global equivariant parameter, one may extend this to a map 
\begin{equation}\label{eq-global-equivariant-parameter-map}
 \mathcal{BPS}^{\alpha}_{\mathcal{X}} \otimes \mathrm{H}^*(\mathrm{B} \mathbb{G}_{\mathrm{m}}^{\dim F})_{\mathrm{vir}} \to p_{\alpha, *} \varphi_{\mathcal{X}_{\alpha}}.
\end{equation}

We now explain the construction of the $(-1)$-shifted symplectic version of the cohomological PBW map \eqref{eq-cohint-map}.
To explain this, choose a chamber $\sigma \subset F$ in the cotangent arrangement and consider the following commutative diagram:
\[\begin{tikzcd}
	& {\mathcal{X}_{\sigma}^+} \\
	{\mathcal{X}_{\alpha}} && {\mathcal{X}} \\
	{X_{\alpha}} && {X.}
	\arrow["{\gr_{\sigma}}"', from=1-2, to=2-1]
	\arrow["{\ev_{\sigma}}", from=1-2, to=2-3]
	\arrow["{p_{\alpha}}"', from=2-1, to=3-1]
	\arrow["p", from=2-3, to=3-3]
	\arrow["g_{\alpha}", from=3-1, to=3-3]
\end{tikzcd}\]
In \cite[\S 8.1.6]{bu2025cohomology}, the authors constructed the \textit{cohomological Hall induction} map
\[
 *^{\mathrm{Hall}}_{\sigma} \colon g_{\alpha, *} p_{\alpha, *} \varphi_{\mathcal{X}_{\alpha}} \to p_* \varphi_{\mathcal{X}} 
\]
building on \cite[Theorem B]{kinjo2024cohomological}. By combining this map with \eqref{eq-global-equivariant-parameter-map}, we obtain a natural map
\[
   \Gamma_{\sigma} \colon g_{\alpha, *} \mathcal{BPS}_{X}^{\alpha} \otimes \mathrm{H}^*(\mathrm{B} \mathbb{G}_{\mathrm{m}}^{\dim F})_{\mathrm{vir}}  \to p_* \varphi_{\mathcal{X}}.
\]
For each non-degenerate face $(F, \alpha) \in \Face^{\nd}(\CX)$, we choose a cone $\sigma_{\alpha} \subset F$ in the cotangent arrangement and define the cohomological PBW map as
\begin{equation}\label{eq-coh-int--1}
 \Psi = \Psi_{\mathcal{X}} \colon \bigoplus_{(F, \alpha) \in \Face^{\nd}(\CX)}  \left( g_{\alpha, *} \mathcal{BPS}_{X}^{\alpha} \otimes \mathrm{H}^*(\mathrm{B} \mathbb{G}_{\mathrm{m}}^{\dim F})_{\mathrm{vir} } \right)^{\Aut(\alpha)} \xrightarrow[]{\sum \Gamma_{\sigma_{\alpha}}} p_* \varphi_{\mathcal{X}}.
\end{equation}
\begin{theorem}\label{thm-main--1}
    The cohomological PBW map \eqref{eq-coh-int--1} is an isomorphism.
\end{theorem}

\begin{proof}
    This follows from \Cref{theorem:main_theorem} and the discussion in \cite[\S 9.3.4]{bu2025cohomology}.
\end{proof}

\subsection{BPS decomposition for the character stacks of $3$-manifolds}
We now specialize \Cref{thm-main--1} to the case of character stacks on compact oriented $3$-manifolds.
Let $M$ be a compact oriented $3$-manifold and $G$ be a connected reductive group.
We let $\mathcal{L}\mathrm{oc}_G(M)$ denote the moduli stack of $G$-local systems on $M$.
By \cite[\S 2.1]{pantev2013shifted}, $\mathcal{L}\mathrm{oc}_G(M)$ is $(-1)$-shifted symplectic, and it is shown in \cite[Theorem 3.45]{naef2023torsion} that $\mathcal{L}\mathrm{oc}_G(M)$ admits a canonical orientation.
Also, it follows from \cite[\S 3.1.3(iii), Lemma 9.1.3]{bu2025cohomology} that the stack $\mathcal{L}\mathrm{oc}_G(M)$ satisfies the assumptions in \Cref{para:assumptions-stacks}.
Therefore the BPS decomposition theorem holds for $\mathcal{L}\mathrm{oc}_G(M)$.
We explicitly express the formula following \cite[Theorem 10.3.30]{bu2025cohomology}.

For a Levi and parabolic subgroups $L \subset P \subset G$, consider the following commutative diagram
\[\begin{tikzcd}
	& {\mathcal{L}\mathrm{oc}_{P}(M)} \\
	{\mathcal{L}\mathrm{oc}_{L}(M)} && {\mathcal{L}\mathrm{oc}_{G}(M)} \\
	{\mathrm{Loc}_{L}(M)} && {\mathrm{Loc}_{G}(M).}
	\arrow["{\gr_P}"', from=1-2, to=2-1]
	\arrow["{\ev_P}", from=1-2, to=2-3]
	\arrow["{p_L}"', from=2-1, to=3-1]
	\arrow["{p_G}", from=2-3, to=3-3]
	\arrow["{g_L}"', from=3-1, to=3-3]
\end{tikzcd}\]
Here, vertical maps are good moduli morphisms.
We set $c_L \coloneqq \dim \mathrm{Z}(L)$ and define
\[
     \mathcal{BPS}_{\mathrm{Loc}_L(M)} \coloneqq \mathcal{H}^{0}(p_{L, *} \varphi_{\mathcal{L}\mathrm{oc}_L(M)} \otimes \mathscr{L}^{- c_L/2}).
\]
Let $W_G(L) = \mathrm{N}_G(L) / L$ denote the relative Weyl group. 
Then $W_G(L)$ acts on $\mathrm{Loc}_L(M)$ and the BPS sheaf is equivariant with respect to this action.
With these preparations, we can state the explicit form of the BPS decomposition theorem for character stacks:
\begin{theorem}\label{thm:main-characterstack}
    Let $M$ be a compact oriented $3$-manifold and $G$ a connected reductive group. There exists a natural isomorphism
    \[
     \bigoplus_{\substack{L \subset G : \\ \textnormal{Levi subgroups}}} \left( g_{L, *} (\mathcal{BPS}_{\mathrm{Loc}_L(M)}) \otimes \mathrm{H}^*( \mathrm{B} \mathrm{Z}(L))_{\mathrm{vir}} \right)^{W_G(L)} \cong p_{G, *} \varphi_{\mathcal{L}\mathrm{oc}_G(M)}.
    \]
    Further, the map from the left to right is given by the parabolic induction map as constructed in \cite[Corollary 9.12]{kinjo2024cohomological}.
\end{theorem}

\subsection{BPS decomposition for $0$-shifted symplectic stacks}

We now use the dimensional reduction theorem \cite[Theorem 4.14]{kinjo2022dimensional} due to the second named author to prove the BPS decomposition theorem for $0$-shifted symplectic stacks.
Throughout the section, we let $\mathcal{Y}$ denote the $0$-shifted symplectic stack satisfying the assumptions in \Cref{para:assumptions-stacks}.
We note that the symmetricity of $\mathcal{Y}$ is automatic from the existence of the $0$-shifted symplectic structure.

We let $\mathcal{X} = \mathrm{T}^*[-1] \mathcal{Y} \coloneqq \mathrm{Tot}_{\mathcal{Y}}(\mathbb{L}_{\mathcal{Y}}[-1])$ denotes the $(-1)$-shifted cotangent stack of $\mathcal{Y}$.
The stack $\mathcal{X}$ is $(-1)$-shifted symplectic stack together with a canonical orientation. Further, it follows from \cite[Corollary 4.2.10]{bu2025cohomology} that $\mathcal{X}$ also satisfies the assumptions in \Cref{para:assumptions-stacks}.
Therefore the BPS decomposition theorem holds for $\mathcal{X}$.

For each face $(F, \alpha) \in \Face^{\nd}(\mathcal{Y}) \cong \Face^{\nd}(\mathcal{X})$, consider the following commutative diagram:
\[\begin{tikzcd}
	{\mathcal{X_{\alpha}}} & {\mathcal{Y}_{\alpha}} \\
	X_{\alpha} & Y_{\alpha}
	\arrow["\pi_{\alpha}", from=1-1, to=1-2]
	\arrow["{\tilde{p}}_{\alpha}"', from=1-1, to=2-1]
	\arrow["p_{\alpha}", from=1-2, to=2-2]
	\arrow["{\bar{\pi}_{\alpha}}"', from=2-1, to=2-2]
\end{tikzcd}\]
where the vertical maps are good moduli morphisms and the horizontal maps are the projections.
We define the BPS sheaf on $Y_{\alpha}$ by
\[
 \mathcal{BPS}_{Y}^{\alpha} \coloneqq \bar{\pi}_{\alpha, *} \mathcal{BPS}_{X}^{\alpha} \otimes \mathscr{L}^{\dim F /2}. 
\]
It is shown in \cite[Theorem 7.2.15]{bu2025cohomology} that the BPS sheaf $\mathcal{BPS}_{Y}^{\alpha}$ is a pure Hodge module.
Recall from \cite[Theorem 4.14]{kinjo2022dimensional} (see also \cite[(6.2.1.4), Lemma 6.1.13]{bu2025cohomology}) that there is a natural isomorphism of $\Aut(\alpha)$-equivariant mixed Hodge module complexes
\begin{equation}\label{eq-dim-red}
 \pi_{\alpha, *} \varphi_{\mathcal{X}_{\alpha}} \cong  ( \mathbb{D} \mathbb{Q}_{\mathcal{Y}_{\alpha}} \otimes \mathscr{L}^{\vdim \mathcal{Y}_{\alpha} / 2}) \otimes \mathrm{sgn}_{\alpha},
\end{equation}
where $\mathrm{sgn}_{\alpha}$ is the cotangent sign representation of $\mathrm{Aut}(\alpha)$.
From now on, we will equip $\mathcal{BPS}_{Y}^{\alpha}$ with an $\mathrm{Aut}(\alpha)$-equivariant structure twisted by $\mathrm{sgn}_{\alpha}$, so that the natural map
\[
    \mathcal{BPS}_{Y}^{\alpha} \to p_* \mathbb{D} \mathbb{Q}_{\mathcal{Y}_{\alpha}} \otimes \mathscr{L}^{\vdim \mathcal{Y}_{\alpha}/2}
\]
is $\mathrm{Aut}(\alpha)$-equivariant.

With this preparation, we can state the BPS decomposition theorem for $0$-shifted symplectic stacks.
Fix a chamber $\sigma \subset F$ with respect to a cotangent arrangement. Consider the following commutative diagram 
\[\begin{tikzcd}
	& {\mathcal{Y}_{\sigma}^+} \\
	{\mathcal{Y}_{\alpha}} && {\mathcal{Y}} \\
	{Y_{\alpha}} && {Y.}
	\arrow["{\gr_{\sigma}}"', from=1-2, to=2-1]
	\arrow["{\ev_{\sigma}}", from=1-2, to=2-3]
	\arrow["{p_{\alpha}}"', from=2-1, to=3-1]
	\arrow["p", from=2-3, to=3-3]
	\arrow["g_{\alpha}", from=3-1, to=3-3]
\end{tikzcd}\]
Using \eqref{eq-dim-red}, the pushforward of the map \eqref{eq-coh-int--1} along the projection $\bar{\pi} \colon X \to Y$ can be rewritten as
\begin{equation}\label{eq-coh-int-0}
    \bar{\Psi}_{\CY}=\overline{\pi}_*\Psi_{\CX} \colon \bigoplus_{(F, \alpha) \in \Face^{\nd}(\CY)}  \left( g_{\alpha, *} \mathcal{BPS}_{Y}^{\alpha} \otimes \mathrm{H}^*(\mathrm{B} \mathbb{G}_{\mathrm{m}}^{\dim F}) \otimes \mathrm{sgn}_{\alpha} \right)^{\Aut(\alpha)} \xrightarrow[]{} p_* \mathbb{D} \mathbb{Q}_{\mathcal{Y}} \otimes \mathscr{L}^{\vdim \mathcal{Y}/2}.
\end{equation}
Theorem \ref{thm-main--1} implies the following:
\begin{theorem}\label{thm-main0}
    The cohomological PBW map $\bar{\Psi}_{\CY}$ \eqref{eq-coh-int-0} is an isomorphism.
\end{theorem}

\begin{remark}
    If there is an auxiliary torus $T'$ acting on $\mathcal{Y}$ (which does not necessary preserve the symplectic form), the BPS sheaf $\mathcal{BPS}_{Y}^{\alpha} $ is naturally $T'$-equivariant and the map \eqref{eq-coh-int-0} is also $T'$-equivariant.
    In particular, it induces an isomorphism on the $T'$-equivariant cohomology groups. 
\end{remark}

\section{Applications: purity and Kirwan surjectivity}

In this section, we will apply the BPS decomposition theorems to Halpern-Leistner's purity conjecture for $0$-shifted symplectic stacks \cite[Conjecture 4.4]{halpern2015theta} and Kirwan surjectivity statements for $0$-shifted and $(-1)$-shifted symplectic stacks.

\subsection{Purity conjecture}

Here, we explain the BPS decomposition theorem for $0$-shifted symplectic stacks (= \Cref{thm-main0}) implies the purity properties of the Borel--Moore homology of $0$-shifted symplectic stacks.
We first explain a relative purity result.

\begin{corollary}\label{cor-relative-purity}
    Let $\CY$ be a $0$-shifted symplectic stack with affine diagonal admitting a good moduli space $p \colon \CY \to Y$. Then the complex $p_* \mathbb{D} \mathbb{Q}_{\mathcal{Y}}$ is pure. 
\end{corollary}

\begin{proof}
    Since the statement is local on $Y$, we may use the local structure theorem of Alper--Hall--Rydh \cite[Theorem 4.12]{alper2020luna} to reduce to the case $\CY = \Spec ( A ) / G$ for some reductive group $G$.
    In this case, \cite[\S 2.3.7]{bu2025cohomology} and \cite[Lemma 9.1.3]{bu2025cohomology} imply that $\CY$ satisfies the assumptions in \Cref{para:assumptions-stacks}.
    In particular, using \Cref{thm-main0}, we see that the cohomological PBW map \eqref{eq-coh-int-0} is an isomorphism.
    Since the Hodge module $\mathcal{BPS}_{Y}^{\alpha} $ is pure by \cite[Theorem 7.2.15]{bu2025intrinsic}, we conclude that  $p_* \mathbb{D} \mathbb{Q}_{\mathcal{Y}}$ is pure as desired.
\end{proof}

By passing to the global section, we obtain the following statement, confirming Halpern-Leistner's purity conjecture \cite[Conjecture 4.4]{halpern2015theta}.

\begin{corollary}\label{cor-purity-conj}
    Let $\CY$ be a $0$-shifted symplectic stack with affine diagonal admitting a good moduli space $p \colon \CY \to Y$.
    Assume further that $\CY$ admits a $\mathbb{G}_{\mathrm{m}}$-action such that the induced $\mathbb{G}_{\mathrm{m}}$-action on $Y$ is contracting and $Y^{\mathbb{G}_{\mathrm{m}}}$ is proper.
    Then the Borel--Moore homology $\mathrm{H}_*^{\mathrm{BM}}(\CY)$ is pure.
\end{corollary}

This is a generalization of \cite[Theorem~A]{davison2021purity} which deals with the case of \emph{linear} $0$-shifted symplectic stacks -- that is stacks of objects in $2$-Calabi--Yau categories.
The following stacks are important examples satisfying the hypothesis of \Cref{cor-purity-conj}:

\begin{enumerate}
    \item Let $G$ be a reductive group and $C$ be a smooth projective curve. Then the moduli stack $\CH_{G}^{\mathrm{ss}}$ of semistable $G$-Higgs bundles on $C$ satisfies the hypotheses in  \Cref{cor-purity-conj}.  \vspace{3pt}
    \item Let $G$ be a reductive group and $V$ be a symplectic $G$-representation. Let $\mu \colon V \to \mathfrak{g}^*$ be the moment map. Then the Hamiltonian reduction $\mu^{-1}(0) / G$ satisfies the hypotheses in  \Cref{cor-purity-conj}.
\end{enumerate}

We note that \Cref{cor-purity-conj} for $\CY = \CH_{G}^{\mathrm{ss}}$ can also be proved using \cite[Theorem 9.4.3]{bu2025cohomology}, where the authors prove the BPS decomposition theorem (= \Cref{thm-main0}) under a stronger hypothesis called \emph{almost-orthogonality}, which is satisfied for $\CH_{G}^{\mathrm{ss}}$ and $\mu^{-1}(0) / G$ when $V$ is the cotangent representation, i.e., it is of the form $V = W \oplus W^{\vee}$.
On the other hand, \Cref{cor-purity-conj} for $\CY = \mu^{-1}(0) / G$ for general symplectic representations $V$ requires the general BPS decomposition theorem.

\subsection{Kirwan surjectivity}\label{ssec-Kirwan}

We will apply \Cref{thm-main--1} to the Kirwan surjectivity statement for the critical cohomology of $(-1)$-shifted symplectic stacks and the Borel--Moore homology of $0$-shifted symplectic stacks.
Kirwan surjectivity refers to the phenomenon that the restriction map for cohomological invariants of stacks to the semistable locus is a surjective morphism.

Let $\CX$ be a $(-1)$-shifted symplectic stack satisfying the assumptions \Cref{para:assumptions-stacks} and $\CL$ be a line bundle on $\CX$.
Let $\CX^{\mathrm{ss}} \subset \CX$ denote the open substack consisting of semistable points with respect to $\CL$: see \cite[\S 4.1]{halpern2020derived} and \cite[\S 7.3.1]{bu2025cohomology} for details.
The stack $\CX^{\mathrm{ss}}$ admits a good moduli space $p^{\mathrm{ss}} \colon \CX^{\mathrm{ss}} \to X^{\mathrm{ss}}$ and satisfies the assumptions \Cref{para:assumptions-stacks}.
Consider the following commutative diagram:
\[\begin{tikzcd}
	{\CX^{\mathrm{ss}}} & \CX \\
	{X^{\mathrm{ss}}} & {X.}
	\arrow["\iota", from=1-1, to=1-2]
	\arrow["{p^{\mathrm{ss}}}"', from=1-1, to=2-1]
	\arrow["p", from=1-2, to=2-2]
	\arrow["q"', from=2-1, to=2-2]
\end{tikzcd}\]
We have the following Kirwan surjectivity statement.

\begin{theorem}\label{thm-kirwan--1}
    Let $\CX$ be an oriented $(-1)$-shifted symplectic stack satisfying the assumptions \Cref{para:assumptions-stacks}.
    Let $\CL$ be a line bundle and $\CX^{\mathrm{ss}} \subset \CX$ be the open substack of semistable points.
    Assume further that there exists a $\Theta$-stratification $\CX = \bigcup_{c} \CX_{\leq c}$ with $\CX_{\leq 0} = \CX^{\mathrm{ss}}$.
    Then the restriction map
    \[
     p_* \varphi_{\CX} \to q_* p^{\mathrm{ss}}_* \varphi_{\CX^{\mathrm{ss}}}    
    \]
    is a split surjection. In particular, the restriction map between the global section
    \[
     \mathrm{H}^*(\CX, \varphi_{\CX}) \to    \mathrm{H}^*(\CX^{\mathrm{ss}}, \varphi_{\CX^{\mathrm{ss}}})  
    \]
    is surjective.
\end{theorem}

\begin{proof}
    For a non-degenerate face $(F, \alpha) \in \Face^{\mathrm{nd}}(\CX^{\mathrm{ss}})$, we let $(F, \bar{\alpha}) \in \Face^{\mathrm{nd}}(\CX)$ denote the image.
    Consider the following commutative diagram of good moduli spaces
    \[\begin{tikzcd}
        {(X^{\mathrm{ss}})_{\alpha}} & {X^{\mathrm{ss}}} \\
        {X_{\bar{\alpha}}} & {X.}
        \arrow["{g^{\mathrm{ss}}_{\alpha}}", from=1-1, to=1-2]
        \arrow["{q_{\alpha}}"', from=1-1, to=2-1]
        \arrow["q", from=1-2, to=2-2]
        \arrow["{g_{\bar{\alpha}}}"', from=2-1, to=2-2]
    \end{tikzcd}
    \]
    By \cite[Proposition 2.3.4, Lemma 4.1.19]{halpern2014structure}, the open substack $(\CX^{\mathrm{ss}})_{\alpha} \subset \CX^{\mathrm{ss}}_{\bar{\alpha}}$ is the semistable locus with respect to the line bundle $\CL |_{\CX^{\mathrm{ss}}_{\bar{\alpha}}}$.
    In particular, \cite[Proposition 7.3.6]{bu2025intrinsic} implies an isomorphism
    \begin{equation}\label{eq-wall-crossing}
      \mathcal{BPS}_{X}^{\bar{\alpha}} \cong q_{\alpha, *} \mathcal{BPS}^{\alpha}_{X^{\mathrm{ss}}} .   
    \end{equation}

    Take a chamber $\sigma \subset F$ with respect to the cotangent arrangement for $\CX^{\mathrm{ss}}$ and $\bar{\sigma} \subset \sigma$ be a subcone which is a chamber with respect to the cotangent arrangement for $\CX$.
    Now consider the following diagram:
    \[\begin{tikzcd}
        {\bigoplus_{(F, \alpha) \in \Face^{\nd}(\CX^{\mathrm{ss}})}  g_{\bar{\alpha}, *}\mathcal{BPS}^{\bar{\alpha}}_{X} \otimes \mathrm{H}^* (\mathrm{B} \mathbb{G}_{\mathrm{m}}^{\dim F})} & {\bigoplus_{(F, \alpha) \in \Face^{\nd}(\CX^{\mathrm{ss}})}   g_{\bar{\alpha}, *} p_{\bar{\alpha}, *}\varphi_{\CX_{\bar{\alpha}}}} & {\varphi_{\CX}} \\
        {\bigoplus_{(F, \alpha) \in \Face^{\nd}(\CX^{\mathrm{ss}})}  q_* g_{\alpha, *}^{\mathrm{ss}} \mathcal{BPS}^{\alpha}_{X^{\mathrm{ss}}} \otimes \mathrm{H}^* (\mathrm{B} \mathbb{G}_{\mathrm{m}}^{\dim F})} & {\bigoplus_{(F, \alpha) \in \Face^{\nd}(\CX^{\mathrm{ss}})}  q_* g_{\alpha, *}^{\mathrm{ss}} p^{\mathrm{ss}}_{\alpha, *}\varphi_{(\CX^{\mathrm{ss}})_{{\alpha}}}} & {q_* \varphi_{\CX^{\mathrm{ss}}}.}
        \arrow[from=1-1, to=1-2]
        \arrow[from=2-1, to=2-2]
        \arrow["\cong", "\eqref{eq-wall-crossing}"', from=1-1, to=2-1]
        \arrow["{\sum_{(F,\alpha)} *^{\mathrm{Hall}}_{\bar{\sigma}}}", from=1-2, to=1-3]
        \arrow[from=1-2, to=2-2]
        \arrow[from=1-3, to=2-3]
        \arrow["{\sum_{(F,\alpha)} *^{\mathrm{Hall}}_{\sigma}}", from=2-2, to=2-3]
    \end{tikzcd}\]
    The left square commutes by the proof of \cite[Proposition 7.3.6]{bu2025intrinsic} and the right square commutes by \cite[(6.2.3.3)]{bu2025cohomology}.
    In particular, the outer square commutes.
    The composition of the lower horizontal maps is a split surjection by \Cref{thm-main--1}.
    Therefore we conclude that the right vertical map is a split surjection as desired.
\end{proof}

\begin{remark}
    We expect that the above surjectivity statement holds without the hypotheses on the existence of the good moduli spaces nor the symmetricity condition.
    Such a generalization, together with the integral isomorphism (proved in \cite[Theorem B]{kinjo2024cohomological} and independently in \cite[Theorem 1.2]{descombes2025hyperbolic}), would lead to the wall-crossing formula for the cohomological Donaldson--Thomas invariants, extending Davison--Meinhardt's wall-crossing formula \cite[Theorem B]{davison2020cohomological} for quivers with potentials.
\end{remark}


    

We now turn to the Kirwan surjectivity for the Borel--Moore homology of $0$-shifted symplectic stacks.
Though it can be reduced to \Cref{thm-main--1} using the dimensional reduction theorem \eqref{eq-dim-red}, 
we will give a different proof using the purity statement that applies to a more general setting under weaker hypotheses and implies a stronger version.
We will use the following lemma.

\begin{lemma}\label{lem-homotopy-lemma}
    Let $\CY$ be a $0$-shifted symplectic stack, $\CS \subset \CY$ be a $\Theta$-stratum and $\mathrm{gr}_{\mathcal{S}} \colon \CS \to \CZ$ denote the map to its center.
    Then the map $\mathrm{gr}_{\mathcal{S}}$ is smooth. Further, there is a natural isomorphism 
    \begin{equation}\label{eq-homotopy-lemma}
        \gr_{\CS, *} \mathbb{D} \mathbb{Q}_{\CS} \cong \mathbb{D} \mathbb{Q}_{\CZ} \otimes \mathscr{L}^{-  \dim \gr_{\CS}}.
    \end{equation}
\end{lemma}

\begin{proof}
    By \cite[Corollary 5.18]{kinjo2024cohomological}, the correspondence $\CZ \xleftarrow{\gr_{\CS}} \CS \rightarrow \CY$ is a Lagrangian correspondence. In particular, there is a natural equivalence
    \[
     \mathbb{T}_{\mathrm{gr}_{\mathcal{S}}} \simeq \mathbb{L}_{\CS / \CY}[-1].    
    \]
    Since the map $\CS \to \CY$ is a closed immersion, the complex $\mathbb{L}_{\CS / \CY}[-1]$ has Tor-amplitude $(-\infty, 0]$.
    Combined with the above equivalence, we conclude that $\gr_{\CS}$ is smooth as desired.
    
    Let $\sigma_{\CS} \colon \CZ \to \CS$ be the canonical section of $\gr_{\CS}$. The second statement follows from the isomorphisms
    \[
        \gr_{\CS, *} \mathbb{D} \mathbb{Q}_{\CS} \cong \gr_{\CS, *} \gr_{\CS}^* \mathbb{D} \mathbb{Q}_{\CZ} \otimes  \mathscr{L}^{-  \dim \gr_{\CS}} \cong \sigma_{\CS}^* \gr_{\CS}^* \mathbb{D} \mathbb{Q}_{\CZ} \otimes \mathscr{L}^{-  \dim \gr_{\CS}} \cong \mathbb{D} \mathbb{Q}_{\CZ} \otimes \mathscr{L}^{- \dim \gr_{\CS}},
    \]
    where the first isomorphism follows from the smoothness of $\gr_{\CS}$ and the second isomorphism follows from \cite[Corollary 7.3]{kinjo2024cohomological}.
\end{proof}

Let $\CY$ be an $0$-shifted symplectic stack satisfying the assumptions \Cref{para:assumptions-stacks}.
Let $\ell \colon |\mathrm{CL}(\mathcal{Y})| \to \mathbb{Q}$ be a linear form \cite[Definition 2.4.1]{nunez2023refined} and $Q \colon |\mathrm{CL}(\mathcal{X})| \to \mathbb{Q}_{\geq 0}$ be a norm on graded points \textnormal{\cite[Definition 4.1.12]{halpern2014structure}}.
Assume that the numerical invariant $\mu = l / \sqrt{Q}$ defines a $\Theta$-stratification $\CY = \bigcup_{c} \CY_{\leq c}$ \textnormal{\cite[Definition 4.1.3]{halpern2014structure}}.
Let $\CS_c \coloneqq \CY_{\leq c} \setminus \CY_{< c}$ be the $\Theta$-stratum and $\gr_c \colon \CS_c \to \CZ_c$ be the map to its centre.
It follows from \cite[Theorem 2.6.3 (2)]{nunez2023refined} that $\CZ_c$ admits a good moduli space $p_c \colon \CZ_c \to Z_c$, and we have the following commutative diagram:
\[\begin{tikzcd}
	{\CZ_c} & {\CS_c} & \CY \\
	{Z_c} && Y.
	\arrow["{{\sigma_c}}", shift left, from=1-1, to=1-2]
	\arrow["{{p_c}}", from=1-1, to=2-1]
	\arrow["{\gr_c}", shift left, from=1-2, to=1-1]
	\arrow["{{\iota_c}}", from=1-2, to=1-3]
	\arrow["p", from=1-3, to=2-3]
	\arrow["{{q_c}}"', from=2-1, to=2-3]
\end{tikzcd}\]
The map $q_c$ is proper by \cite[Theorem 2.6.3 (4)]{nunez2023refined}.
Write $\CY^{\mathrm{ss}} \coloneqq \CY_{\leq 0}$, $p^{\mathrm{ss}} \coloneqq p_0$ and $q \coloneqq q_0$.
Then we have the following:

\begin{theorem}\label{thm-Kirwan-relative}
    We adopt the notations from the last paragraph.
    \begin{enumerate}
        \item[\textnormal{(1)}] The natural restriction maps
        \[
         p_* \mathbb{D}\mathbb{Q}_{\CY} \to    q_* p_{ *}^{\mathrm{ss}} \mathbb{D}\mathbb{Q}_{\CY^{\mathrm{ss}}}, \quad \mathrm{H}^{\mathrm{BM}}_*(\CY) \to  \mathrm{H}^{\mathrm{BM}}_*(\CY^{\mathrm{ss}})
        \]
        are split surjections.
        \item[\textnormal{(2)}] Set $d_c \coloneqq \dim (\CS_c / \CZ_c)$. Then there are decompositions
        \[
         p_*   \mathbb{D}\mathbb{Q}_{\CY} \cong \bigoplus_c \left( q_{c, *} p_{c, *} \mathbb{D}\mathbb{Q}_{\CZ_c} \otimes \mathscr{L}^{- d_c} \right), \quad \mathrm{H}^{\mathrm{BM}}_*(\CY) \cong \bigoplus_c \left( \mathrm{H}^{\mathrm{BM}}_*(\CZ_c)  \otimes \mathscr{L}^{- d_c} \right).
        \]
    \end{enumerate}
\end{theorem}

\begin{proof}
    The first statement follows from the purity of $p_* \mathbb{D}\mathbb{Q}_{\CY}$ and $q_* p_{ *}^{\mathrm{ss}} \mathbb{D}\mathbb{Q}_{\CY^{\mathrm{ss}}}$ proved in \Cref{cor-relative-purity}.
    Now we prove the second statement. Let $j_{\leq c} \colon \CY_{\leq c} \hookrightarrow \CY$ and $j_{< c} \colon \CY_{< c} \hookrightarrow \CY$ denote the natural open immersion.
    Consider the fibre sequence
    \begin{equation}\label{eq-fibre-seq}
      (p \circ \iota_c)_*  \mathbb{D} \mathbb{Q}_{\CS_c}  \to   (p \circ j_{\leq c})_* \mathbb{D}\mathbb{Q}_{\CY_{\leq c}} \to (p \circ j_{< c})_* \mathbb{D}\mathbb{Q}_{\CY_{< c}}.
    \end{equation}
    We claim the splitting of this fibre sequence and the purity of $(p \circ j_{\leq c})_* \mathbb{D}\mathbb{Q}_{\CY_{\leq c}}$ by induction.
    By the induction hypothesis, we may assume that the complex $(p \circ j_{< c})_* \mathbb{D}\mathbb{Q}_{\CY_{< c}}$ is pure.
    On the other hand, we have an isomorphism
    \[
        (p \circ \iota_c)_*  \mathbb{D} \mathbb{Q}_{\CS_c} \cong q_{c, *}  p_{c, *} \gr_{c, *}  \mathbb{D} \mathbb{Q}_{\CS_c} \xrightarrow[\cong]{\textnormal{\eqref{eq-homotopy-lemma}}} q_{c, *}  p_{c, *} \mathbb{D} \mathbb{Q}_{\CZ_c} \otimes \mathscr{L}^{ d_c}.
    \]
    Since $q_c$ is proper, using \Cref{cor-relative-purity} for $\CZ_c$, we see that the complex $q_{c, *}  p_{c, *} \mathbb{D} \mathbb{Q}_{\CZ_c} $ is pure. In particular, $(p \circ \iota_c)_*  \mathbb{D} \mathbb{Q}_{\CS_c} $ is pure.
    Therefore the fibre sequence \eqref{eq-fibre-seq} splits, and also we conclude that $ (p \circ j_{\leq c})_* \mathbb{D}\mathbb{Q}_{\CY_{\leq c}} $ is pure.
\end{proof}

\begin{remark}
    A similar Kirwan surjectivity statement for the Borel--Moore homology of $0$-shifted symplectic stacks, together with a sketch of proof, was obtained by Halpern-Leistner \cite[Corllary 4.1]{halpern2015theta} via a categorical argument.
    His version works more generally for derived stacks $\CY$ with $\mathbb{T}_{\CY} \simeq \mathbb{L}_{\CY}$, and more importantly, also for stacks without good moduli spaces, as long as $\CY$ is a derived quotient stack.

    Our version, by contrast, has an advantage of working relatively over the good moduli spaces and of applying also to non-quotient stacks.
    We also have a version of the Kirwan surjectivity for stacks without good moduli spaces as explained below.
\end{remark}

We prove a different version of the Kirwan surjectivity statement for $0$-shifted symplectic stacks, which works for stacks without good moduli spaces, but instead has a certain properness properties.

\begin{theorem}
    Let $\CY$ be a $0$-shifted symplectic stack with affine diagonal. Assume that $\CY$ admits a $\Theta$-stratification $ \CY_{\leq c}  \subset \CY$ with the following properties:
    \begin{enumerate}
        \item[\textnormal{(1)}] The stack $\CX_{\leq c}$ has quasi-compact connected components for each $c$.
        \item[\textnormal{(2)}] Let $\CS_c$ be a $\Theta$-stratum and $\CZ_c$ be the centre. Then $\CZ_{c}$ admits a good moduli space $Z_c$.
        \item[\textnormal{(3)}] There is a $\mathbb{G}_{\mathrm{m}}$-action on $\CZ_c$ such that the fixed point $Z_c^{\mathbb{G}_{\mathrm{m}}}$ is proper.
    \end{enumerate}
    Then The restriction map $\mathrm{H}^{\mathrm{BM}}_*(\CY) \to \mathrm{H}^{\mathrm{BM}}_*(\CY^{\mathrm{ss}})$ is surjective and there is a direct product decomposition
    \[
        \mathrm{H}^{\mathrm{BM}}_*(\CY) \cong \prod_c \left( \mathrm{H}^{\mathrm{BM}}_*(\CZ_c)  \otimes \mathscr{L}^{- d_c} \right).
    \]
\end{theorem}

\begin{proof}
    \Cref{cor-purity-conj} and \Cref{lem-homotopy-lemma} imply that the Borel--Moore homology $\mathrm{H}^{\mathrm{BM}}_*(\CS_c)$ is pure.
    In particular, arguing as the proof of \Cref{thm-Kirwan-relative}, we see that the fibre sequence
    \[
    \mathrm{H}^{\mathrm{BM}}_*(\CS_c) \to \mathrm{H}^{\mathrm{BM}}_*(\CX_{\leq c}) \to \mathrm{H}^{\mathrm{BM}}_*(\CX_{<c})    
    \]
    splits. Hence we obtain the desired statement.
\end{proof}

Applying the above theorem to the moduli stack of $G$-Higgs bundles, we obtain the following:

\begin{corollary}
    Let $C$ be a smooth projective curve, $G$ a reductive algebraic group, $\CH_G$ the moduli stack of $G$-Higgs bundles on $C$ and $\CH_G^{\mathrm{ss}} \subset \CH_G$ the semistable locus.
    Then the restriction map 
    \[
        \mathrm{H}^{\mathrm{BM}}_*(\CH_G) \to  \mathrm{H}^{\mathrm{BM}}_*(\CH_G^{\mathrm{ss}}) 
    \]
    is surjective.
\end{corollary}

We note that, in contrast, \cite{cliff2018kirwan} shows that the Kirwan map \emph{fails} to be surjective for the moduli stack of $G$-Higgs bundles when one works with ordinary cohomology.

\appendix

\section{Lemmas on homological algebra}

In this appendix, we prove two lemmas on homological algebra used in the proof of \Cref{thm-cohint-satisfy}.

\begin{lemma}\label{lem-injective}
    Let $X$ be an algebraic variety. Let $f \colon E_1 \to E_2$ be a morphism of pure Hodge complexes bounded below.
    Assume that the induced map on the perverse cohomology $\pH(f) \colon \pH(E_1) \to \pH(E_2)$ is injective. Then:
    \begin{enumerate}
        \item The map $\mathrm{H}^*(f) \colon \mathrm{H}^*(E_1) \to \mathrm{H}^*(E_2)$ induced on global sections is injective.
        \item Let $C$ be the cofibre of the map $f$. Then, there exists a short exact sequence
        \[
           0 \to \mathrm{H}^*(E_1) \to \mathrm{H}^*(E_2) \to \mathrm{H}^*(C) \to 0.
        \]
    \end{enumerate}
\end{lemma}

\begin{proof}
    Since $E_1$ and $E_2$ are pure, the adjoint maps with respect to the perverse truncation functors
    \[
     \tau_{\mathfrak{P}}^{\leq n} (E_1) \to E_1, \quad \tau_{\mathfrak{P}}^{\leq n} (E_2) \to E_2   
    \]
    are split injections. In particular, they induce filtrations
    \[
     \mathfrak{P}^{\leq n} \mathrm{H}^*(E_1) \subset  \mathrm{H}^*(E_1), \quad   \mathfrak{P}^{\leq n} \mathrm{H}^*(E_2) \subset  \mathrm{H}^*(E_2)
    \]
    which we will call the perverse filtrations.
    There exist natural isomorphisms
    \[
     \gr_{\mathfrak{P}} \mathrm{H}^*(E_1)  \cong \mathrm{H}^*(\pH(E_1)), \quad  \gr_{\mathfrak{P}} \mathrm{H}^*(E_2)  \cong \mathrm{H}^*(\pH(E_2))
    \]
    which identify $\gr_{\mathfrak{P}} \mathrm{H}^*(f)$ and $\mathrm{H}^*(\pH(f))$.
    
    Now assume that $\mathrm{H}^*(f)$ is not injective and take an element $a \in \mathfrak{P}^{\leq n} \mathrm{H}^*(E_1) \setminus \mathfrak{P}^{\leq n - 1} \mathrm{H}^*(E_1)$ such that $\mathrm{H}^*(f)(a) = 0$.
    Let $0 \neq \bar{a} \in \gr_{\mathfrak{P}}^n(\mathrm{H}^*(E_1))$ be the image of $a$. Then we have $\gr_{\mathfrak{P}} \mathrm{H}^*(f)(\bar{a}) = 0$.
    In particular, $\gr_{\mathfrak{P}} \mathrm{H}^*(f)$ is not injective, hence neither is $\mathrm{H}^*(\pH(f))$ thanks to the identification of these maps.
    On the other hand, the injectivity of $\pH(f)$ together with the purity of $E_1$ and $E_2$ imply the split injectivity of $\pH(f)$, that is $\pH(f)$ admits a retraction.
    In particular, $\mathrm{H}^*(\pH(f))$ also admits a retraction, and so is injective, which leads to the contradiction.
    This proves (1). The second claim follows from (1) and the long exact sequence associated with the fibre sequence $E_1 \to E_2 \to C$.
\end{proof}

\begin{lemma}\label{lem-surjective}
    Let $X$ be an algebraic variety, $x \in X$ be a point and $f \colon E_1 \to E_2$ be a morphism of pure Hodge complexes bounded below.
    Assume that $f$ is an isomorphism over $X \setminus \{ x \}$ and the map induced on the global section $\mathrm{H}^*(f)$ is surjective.
    Then the map induced on the perverse cohomology $\pH(f)$ is a split surjection.
\end{lemma}

\begin{proof}
    Using the purity of $E_1$, we may take a direct summand $K \subset E_1$ such that $K$ does not contain a shift of the constant sheaf supported over $\{ x \}$ as a direct summand and $K |_{X \setminus \{ x \}} \cong E_1 |_{X \setminus \{ x \}}$.
    We write
    \[
     C_1 \coloneqq \cofib(K \to E_1), \quad   C_2 \coloneqq \cofib(K \to E_1 \to E_2).  
    \]
    Let $g \colon C_1 \to C_2$ be the naturally induced map.
    Using \Cref{lem-injective} (2), we obtain a map of short exact sequences
    \[\begin{tikzcd}
        0 & {\mathrm{H}^*(K)} & {\mathrm{H}^*(E_1)} & {\mathrm{H}^*(C_1)} & 0 \\
        0 & {\mathrm{H}^*(K)} & {\mathrm{H}^*(E_2)} & {\mathrm{H}^*(C_2)} & {0.}
        \arrow[from=1-1, to=1-2]
        \arrow[from=1-2, to=1-3]
        \arrow[equals, from=1-2, to=2-2]
        \arrow[from=1-3, to=1-4]
        \arrow["{\mathrm{H}^*(f)}", from=1-3, to=2-3]
        \arrow[from=1-4, to=1-5]
        \arrow["{\mathrm{H}^*(g)}", from=1-4, to=2-4]
        \arrow[from=2-1, to=2-2]
        \arrow[from=2-2, to=2-3]
        \arrow[from=2-3, to=2-4]
        \arrow[from=2-4, to=2-5]
    \end{tikzcd}\]
    The surjectivity of $\mathrm{H}^*(f)$ implies the surjectivity of $\mathrm{H}^*(g)$.
    Since $C_1$ and $C_2$ are supported on $\{ x \}$, it implies the surjectivity of the map $\pH(g) \colon \pH(C_1) \to \pH(C_2)$.
    Since the map $\pH(f)$ is identified with $\id_{\pH(K)} \oplus \pH(g)$, we obtain the desired claim.
\end{proof}

\section{Geometric proof of the cohomological bound}\label{section-geometric-proof}

Here, we give a geometric proof of \Cref{thm-cohomological-bound}.
We use the following statement, which follows from the shuffle presentation of the cohomological Hall induction (see e.g. \cite[Proposition 4.7]{puadurariu2021noncommutative}).

\begin{lemma}\label{lem-commutes-res}
    Let $G$ be a connected reductive group, $T \subset G$ a maximal torus and $V$ be a symmetric representation of $G$.
    Let $\lambda$ be a cocharacter of $T$.
    Then the following diagram commutes up to a non-zero scalar:
    \[\begin{tikzcd}
        {\mathrm{H}^{* + d}(V^{\lambda}   /T)_{\mathrm{vir}}} & {\mathrm{H}^{* + d}(V   /T)_{\mathrm{vir}}} \\
        {\mathrm{H}^*((V \times \mathfrak{g}_{\mathrm{ad}})^{\lambda} /L_{\lambda})_{\mathrm{vir}}} & {\mathrm{H}^*((V \times \mathfrak{g}_{\mathrm{ad}}) / G)_{\mathrm{vir}}.}
        \arrow["{\sum_{w \in W} \left(*_{w(\lambda)}^{\mathrm{Hall}}\right)}", from=1-1, to=1-2]
        \arrow[from=2-1, to=1-1]
        \arrow["{*_{\lambda}^{\mathrm{Hall}}}"', from=2-1, to=2-2]
        \arrow[from=2-2, to=1-2]
    \end{tikzcd}\]
    Here, $d =  \dim T - \dim \mathrm{Z}(G)$, $W$ is the Weyl group of $G$ and the vertical maps are the restriction homomorphisms.
\end{lemma}

Now we provide a proof of \Cref{thm-cohomological-bound}.
First, \Cref{lem-commutes-res} implies a natural isomorphism
\[
 P_{(V \times \mathfrak{g}_{\mathrm{ad}}) / G}  \cong (P_{V / T})^{W} [d].
\]
Since $V$ is an orthogonal $T$-representation, \cite[Theorem 9.1.12]{bu2025cohomology} implies an isomorphism
\[
    P_{V / T} \cong \IH^{\circ, *}(V \GIT T).    
\]
Using the Artin vanishing and $d \geq 0$, we obtain the desired statement.

\printbibliography

\end{document}